\documentclass[reqno,12pt]{amsart}

\usepackage{graphicx,pdfsync,amssymb, latexsym,amsfonts,amsbsy,color,subcaption,esint,mathtools,enumerate,nicefrac,bm}
\usepackage{hyperref}
\usepackage{booktabs}
\usepackage{multirow}
\hypersetup{
colorlinks=true,
linkcolor=blue,
filecolor=blue,
urlcolor=blue,
citecolor=blue,
}

\usepackage{float}
\usepackage{mathrsfs}
\restylefloat{table}

\usepackage{enumitem}
\setitemize{itemsep=0.4em}
\setenumerate{itemsep=0.5em}

\usepackage[utf8]{inputenc}% for special characters
\usepackage[T1]{fontenc}% for special characters
\usepackage[top=1.1in, bottom=1in, left=1in, right=1in]{geometry}
\newtheorem{theorem}{Theorem}[section]
\newtheorem{lemma}[theorem]{Lemma}
\newtheorem{proposition}[theorem]{Proposition}
\newtheorem{definition}[theorem]{Definition}

\newtheorem{remark}[theorem]{Remark}

\newcommand{\thistheoremname}{}
\newtheorem{genericthm}[theorem]{\thistheoremname}

 \newtheorem*{genericthm*}{\thistheoremname}
\newenvironment{namedthm*}[1]
  {\renewcommand{\thistheoremname}{#1}%
   \begin{genericthm*}}
  {\end{genericthm*}}

\usepackage{times}
\usepackage{setspace}
\newcommand{\dd}{\mathop{}\!\mathrm{d}}
\let\del\partial
\newcommand{\supp}{\text{supp}\,}

\newcommand{\RR}{\mathring R}
\newcommand{\R}{\mathbb R}
\newcommand{\QQ}{\mathbb Q}
\newcommand{\ZZ}{\mathbb Z}
\newcommand{\NN}{\mathbb N}
\newcommand{\RRR}{\mathring{\bar R}}
\newcommand{\ootimes}{\mathbin{\mathring{\otimes}}}

\DeclareMathOperator{\tr}{Tr}
\newcommand{\uin}[0]{u^{\textup{in}}}

\newcommand{\vv}[0]{\bar u}

\newcommand{\pp}[0]{\bar p}

\newcommand{\wpq}[1][q+1]{w^{\textup{(p)}}_{#1}}
\newcommand{\wcq}[1][q+1]{w^{\textup{(c)}}_{#1}}

\newcommand{\wtq}[1][q+1]{w^{\textup{(t)}}_{#1}}

\newcommand{\PP}[0]{\mathbb{P}_{\neq 0}}

\newcommand{\TTT}[0]{\mathbb{T}}

\newcommand{\ii}{\textup{i}}

\newcommand{\T}{\textup{T}}

\newcommand{\loc}{\text{loc}}

\newcommand{\Div}{\text{div}}

\newcommand{\uql}{u^{\text{loc}}_q}
\newcommand{\uqnl}{u^{\text{non-loc}}_q}
\newcommand{\uqql}{u^{\text{loc}}_{q+1}}
\newcommand{\uqqnl}{u^{\text{non-loc}}_{q+1}}
\newcommand{\ubql}{\bar{u}^{\text{loc}}_q}
\newcommand{\ubqnl}{\bar{u}^{\text{non-loc}}_q}
\newcommand{\uq}{u_q}
\newcommand{\ulql}{u^{\text{loc}}_{\ell_q}}
\newcommand{\ulqnl}{u^{\text{non-loc}}_{\ell_q}}
\newcommand{\Rem}{R^{\text{rem}}_q}
\newcommand{\LinfB}[1][q+1]{\widetilde{L}^\infty_T B^{\frac{1}{2}}_{2,1}}
\newcommand{\LoB}[1][q+1]{{L}^1_t B^{1/2}_{2,1}}
\newcommand{\LoBt}[1][q+1]{\widetilde {L}^1_t B^{5/2}_{2,1}}
\allowdisplaybreaks[4]
\numberwithin{equation}{section}

\begin{document}

\title{Sharp non-uniqueness  for the Navier-Stokes
equations in $\R^3$}

\author{Changxing Miao}
\address[Changxing Miao]{Institute  of Applied Physics and Computational Mathematics, Beijing, China.}

\email{miao\_changxing@iapcm.ac.cn}

\author{Yao Nie}

\address[Yao Nie]{School of Mathematical Sciences and LPMC, Nankai University, Tianjin, China.}

 \email{nieyao@nankai.edu.cn}

\author{Weikui Ye}

\address[Weikui Ye]{School of Mathematical Sciences, South China Normal University, Guangzhou,  China}

 \email{904817751@qq.com}

%    General info
%\subjclass[2020]{76D03, 35Q35}

\date{\today}

\begin{abstract}In this paper, we prove a sharp and strong non-uniqueness for a class of weak solutions to the incompressible  Navier-Stokes equations in $\R^3$. To be more precise, we exhibit the non-uniqueness result in a strong sense, that is, any weak solution  is non-unique in $ L^p([0,T];L^\infty(\R^3))$ with $1\le p<2$. Moreover, this   non-uniqueness result  is sharp with regard to the classical Ladyzhenskaya-Prodi-Serrin criteria at endpoint $(2, \infty)$, which extends the sharp nonuniqueness for the Navier-Stokes equations on torus $\TTT^3$ in  the  recent groundbreaking work (Cheskidov and Luo, Invent. Math., 229 (2022), pp. 987-1054) to the setting of the whole space. The key ingredient is developing a new  iterative scheme that balances the compact support of the Reynolds stress error with the non-compact support of the solution via introducing   incompressible perturbation fluid.
\end{abstract}
\maketitle
\section{Introduction}
In this paper, we consider the Cauchy problem for the  incompressible Navier-Stokes
equations
\begin{equation}\label{NS}\tag{NS}
\left\{ \begin{alignedat}{-1}
   & \del_t u-\Delta u+\Div (u\otimes u)  +\nabla p   =0,  &(t,x)\in [0,T]\times\R^3,
 \\ 
  &\Div\, u = 0, &(t,x)\in  [0,T]\times\R^3,
\end{alignedat}\right.
\end{equation}
associated with initial data $u|_{t=0}=u_0$. Here $u: \R^3\times [0,T]\to\R^3$ denotes the velocity of the incompressible fluid and $p:\R^3\times [0,T]\to\R$ represents the pressure field.

For the incompressible Navier-Stokes equations \eqref{NS}, Leray in  the original work \cite{L} proved the existence of global weak solutions in $C_w([0, T]; L^2(\R^3)) \cap L^2([0, T]; \dot{H}^1(\R^3)) $ with the energy inequality
$$\|u(t)\|_{L^2}^{2}+2\int_{0}^{t}\|\nabla u(s)\|_{L^2}^{2}{\rm d s}\leq \|u_0\|_{L^2}^{2},\quad \forall \,t\in[0,T].$$
And this result was further developed by Hopf \cite{Hopf} in general domains. This class of weak solutions is now called as\textit{ Leray-Hopf weak solutions}.  The existence of Leray-Hopf weak solutions to the equations  \eqref{NS}  has been proven for nearly a century, but the issues of uniqueness and regularity for Leray-Hopf weak solutions of the equations \eqref{NS} remain open.  In order to understand how far we are from  addressing these challenging problems, many researchers are committed to seeking uniqueness or regularity criteria, which serve as sufficient conditions to ensure the uniqueness or regularity of  Leray-Hopf weak solutions. There  has been a variety of uniqueness criteria, such as the serrin type criterion,  which  was established by Prodi \cite{Prodi}, Serrin \cite{Serr}, Lady\v{z}henskaya \cite{Lady} and Kozono-Sohr~\cite{KS}. This shows that Leray-Hopf weak solutions with the same initial data are consistent on $[0, T]$, provided  one of  Leray-Hopf weak solutions $u$ satisfies
\begin{align}\label{LPS}
    u\in  L^p([0,T];L^q(\R^3)), \,\,\text{with}\,\,\frac{2}{p}+\frac{3}{q}\le 1,\,\,  3\le q\le\infty.
\end{align}
The condition \eqref{LPS}  can also serve as a regularity criterion for the Leray-Hopf weak solutions and readers can refer to \cite{ESS, Lady, Prodi} and references therein for more details.

As a matter of fact,  Fabes, Jones and Rivi\`{e}re \cite{FJR} proved that  the condition \eqref{LPS} except for the endpoint case $q=3$ also acts as uniqueness  criterion  for  \textit{very weak solution} of \eqref{NS}.
Later, {Furioli}, {Lemari\`{e}-Rieusset} and {Terraneo} \cite{FLT} showed that the  very weak solution is unique in the functional space $C([0,T];L^3(\R^3))$. Here the  very weak solution is defined as follows.
\begin{definition}[Very weak solution]\label{def} Let $u_0\in L^2(\R^3)$ be divergence-free in the sense of distributions. We say that  $u\in L^2([0, T]\times\R^3)$ is a \emph{very weak solution}  to  the equations~\eqref{NS} if
\begin{itemize}
     \item [(1)] For a.e. $t\in [0,T]$, $u(\cdot, t)$ is divergence-free in the sense of distributions;
     \item [(2)]For all divergence-free test functions $\phi\in\mathcal{D}_T$,\footnote{Let $\mathscr{S}(\R^n)$ denote the space of rapidly decreasing functions on $\R^n$. We denote by $\mathcal{D}_T$ those vector functions $\phi(x,t)=(\phi_1(x,t), \phi_2(x,t), \phi_3(x,t))$ such that $\phi_i(x,t)\in \mathscr{S}(\R^4)$, $\phi_i(x,t)=0$ for $t\ge T$ and $\Div \phi =\sum_{i=1}^3\partial_{x_i}\phi_i(x,t)=0$ for all $t$.}
     \begin{align}\nonumber
\int_0^T \int_{\R^3} (\del_t-\Delta)\,\phi u+\nabla \phi : u\otimes u \dd x \dd t=  -\int_{\R^3} u_0 \phi(0,x)\dd x,
\end{align}
   \end{itemize}
   where $\nabla \phi : u\otimes u=\partial_i\phi_ju_iu_j$ under the Einstein summation convention.
\end{definition}

On the non-uniqueness problems for different types of weak solutions, there have been major progresses in recent years. For the 3D incompressible Navier-Stokes equations in the periodic setting, Buckmaster and Vicol in \cite{BV} made the significant break-through via a ${L^2_x}$-based intermittent convex integration scheme, demonstrating that weak solutions  are not unique in $C([0,T];L^2(\TTT^3))$. Subsequently, Buckmaster, Colombo and  Vicol \cite{BCV} proved that the wild solutions can be generated by  $H^3(\TTT^3)$ initial data, which implies the non-uniqueness of very weak solutions. However, the regularity of these weak solutions is far from the critical functional space $L^p([0,T];L^q(\TTT^3))$ with  $\frac{2}{p}+\frac{3}{q}=1$.  Recently, Cheskidov and Luo in the remarkable  paper \cite{1Cheskidov} proved the nonuniqueness of very weak solutions
in the class $L^p([0,T]; L^{\infty}(\TTT^3))$ for $1\le p<2$ and this result implies the sharpness of the Ladyzhenskaya-Prodi-Serrin criteria $\frac{2}{p}+\frac{3}{q}\le 1$ at the endpoint $(p,q)=(2,\infty)$. A series of works on Euler equations, e.g.  \cite{Buc, BDIS15, BDS, DS17, DRS, DS09, DS14, Ise17, Ise18}, is based on developing the method of convex integration, which was also successfully applied to study other fluid dynamics models,  for instance, the stationary Navier-Stokes equations \cite{Luo},  the transport equations
\cite{BCD, CL21, CL22, MoS, MS}, the MHD equations \cite{2Beekie, LZZ, MY}, the Boussinesq equations \cite{MNY, TZ17, TZ18}. Apart from using convex integration to construct  non-unique weak solutions, mathematicians attempt to use spectral analysis methods to study the non-uniqueness of Leray-Hopf weak solutions to the Navier-Stokes equations.  For example, Jia and  \v{S}ver\'{a}k \cite{JS} provided  the spectral conditions as sufficient conditions for non-uniqueness of Leray-Hopf weak solutions.  Albritton, Bru\'{e} and Colombo \cite{ABC} proved the non-uniqueness of the Leray-Hopf weak solutions of the equations \eqref{NS} with a special forcing term.

To the best of our knowledge, the works (e.g. \cite{BCV, BV, 1Cheskidov}) on the non-uniqueness of weak solutions for the Navier-Stokes equations within the convex integration scheme are established under the periodic setting, and the iterative scheme in these works cannot be directly applied to the case of unbounded domains, such as $\R^3$. On the other hand, it has been shown that all  very weak solutions of the Navier-Stokes equations \eqref{NS} with the same initial data are consistent in $L^2([0,T]; L^{\infty}(\R^3))$, while the   counterexamples in  \cite{1Cheskidov} show the non-uniqnessness of very weak solutions in $L^p([0,T]; L^{\infty}(\TTT^3))$ for $1\le p<2$.  It is natural to ask whether one can show the sharp non-uniqnessness of  very weak solutions  in $L^p([0,T]; L^{\infty}(\R^3))$ for $1\le p<2$. In this paper, we aim to address this question. For convenience, we refer to very weak solution in Definition \ref{def} as weak solution in the remaining part of this paper.

 Before stating our results, we give two classes of  non-uniqueness definitions in functional space $X$ introduced in \cite{1Cheskidov}:
\begin{align}
&\bullet~\text{``Weak non-uniqueness'': there exists a non-unique weak solution in the class $X$.}\notag\\
&\bullet~\text{``Strong non-uniqueness'': any weak solution in the class $X$ is non-unique.}\notag
\end{align}

 Now we are in position to state our main theorem, which reveals a sharp non-uniqueness in a strong sense with regrad to the classical Ladyzhenskaya-Prodi-Serrin criteria at endpoint $(2, \infty)$ in the whole space~$\R^3$.

\begin{theorem}[Sharp and strong non-uniqueness]\label{t:main0}
Let $1\le p<2$ and $T_0>0$. Any weak solution~$u$ of  the equations \eqref{NS}  in $L^p([0, T_0]; L^{\infty}(\R^3))$ is non-unique.

\begin{remark}
Given $1\le p < 2$, Cheskidov and Luo \cite{1Cheskidov} showed the non-uniqueness of a weak solution $u$ to  the
Navier-Stokes equations in $L^p([0, T_0]; L^{\infty}(\TTT^3))$  \textit{if $u$ has at least one interval regularity}. Our result shows the non-uniqueness of any weak solution $u$ in  $L^p([0, T_0]; L^{\infty}(\R^3))$ without  imposing the regularity assumption on $u$. More importantly, in the setting of the whole space, we develop a new iterative scheme via incompressible perturbation fluid, which is of independent interest.

\end{remark}

\end{theorem}

Let us outline the main ideas in $\R^3$. The construction of weak solutions in Theorem \ref{t:main0} is based on an iteration scheme via the technique of convex integration, see e.g. \cite{BCV, BDIS15, BV, 1Cheskidov, DRS, DS09}. Our strategy is reducing Theorem \ref{t:main0} to Proposition~\ref{p:main-prop} below. To prove Proposition~\ref{p:main-prop}, we construct a sequence of approximate solutions to  the so-called Navier–Stokes-Reynolds  system  governed by~\footnote{Here and below,  $v \otimes u:=  (v_i u_j)_{i,j=1}^3$, and the divergence of a $3\times 3$ matrix $M=(M_{ij})_{i,j=1}^3$ is defined by $\Div M$ with components $(\Div M)_j :=    \partial_i M_{ij}$.}
\begin{equation}\label{NSR}\tag{NSR}
\left\{ \begin{alignedat}{-1}
   & \del_t u_q-\Delta u_q+\Div (u_q\otimes u_q) +\nabla p_q   =\Div\RR_q,
 \\
  &\nabla \cdot u_q = 0, \\
\end{alignedat}\right.
\end{equation}
associated with some initial data $\uin\in H^3(\R^3)$.
The \emph{Reynolds stress} $\RR_q$ is trace-free symmetric matrix.

To construct weak solution $u$ in $L^p_tL^\infty_x\cap L^2_{t,x}$, we design the perturbation $w_{q+1}\sim u_{q+1}-u_q$ such that $\|w_{q+1}\|_{L^p_tL^\infty_x\cap L^2_{t,x}}\to 0$ and $\RR_q \to 0$  in an appropriate sense.  A key ingredient in reducing the size of $\RR_{q+1}$  is to construct highly oscillatory  perturbation $w_{q+1}$ so that  the Reynolds stress error $\RR_q$ can be cancelled by the low frequency of $w_{q+1}\otimes w_{q+1}$, which naturally requires that $\RR_q$ tends to~$0$ in $ L^1_{t,x}$ on the iteration scheme.  Unfortunately, compared with the periodic setting, such requirement makes it difficult to construct $L^2_{t,x}$ perturbations in the entire space.  Roughly speaking, the perturbation $w_{q+1}$ in \cite{1Cheskidov} is constructed as
\begin{align*}
w_{q+1}\sim \sum_{k\in \Lambda}a_k\Big({\rm Id}-\frac{\RR_q}{\chi(\RR_{q})}\Big)\chi^{1/2}(\RR_{q})\mathbf{W}_{k}(\sigma x)g(\nu t),
\end{align*}
 where $\mathbf{W}_{k}$ are Mikado flows, $g$ is  temporal concentration function and $\chi:\R^{3\times 3}\times \R^{+}$ is a positive smooth function. Such $\chi$  guarantees $w_{q+1}\in L^2_{t,x} $ for  $\RR_{q}\in L^1_{t,x}$ with compact spatial support, yet fail to ensure $w_{q+1}\in L^2_{t,x} $ when $\RR_{q}\in L^1_{t,x}$ without compact support.

{From a technical perspective, the procedure of constructing perturbation  via the geometric lemma imposes the Reynolds stress error to be compactly support at each iteration step. As we know, the inverse of the divergence operator preserves the periodic property of  functions. This feature naturally  ensures that the support of  Reynolds stress errors display compactness in the periodic torus. However, in the setting of the whole space, when the inverse of the divergence operator acts on a function with compact support, the resulting function may not have compact support. Furthermore, even for smooth initial data with compact support, the solutions to the Navier-Stokes equations may not possess compact support. {These facts  imply that the iterative scheme utilized on $\TTT^3$ is  not enough  for ensuring  Reynolds stress errors with compact support and divergence-form  at each step in $\R^3$. } This compels us to develop a new iterative scheme.}

{As we know, there has no result on the non-uniqueness of weak solutions for the Navier-Stokes equations via the convex integration in $\R^3$, but there has  some progress for the Euler equations. For instance, Isett and Oh in an impressive work \cite{I-O} constructed  $C^{\frac{1}{5}-}_{t,x}$ Euler flow  with compact space-time support in $\R^3$ via a new method of solving the symmetric divergence Equation, which allows them to obtain a new Euler–Reynold stress that is similarly localized in space. Different from their ideas,  we are focused on developing a new iterative scheme to avoid the use of the divergence inverse operator. The main difficulty   that follows is  to ensure compatibility between the Reynolds stress error with compact support and divergence-form, and the non-compact support property of the solution to  \eqref{NSR} at each step. The core idea of overcoming this difficulty is to balance the non-compact support  and non-divergence form parts of the new Reynolds stress by means of incompressible perturbation fluid  $\wtq$. }

To put it roughly, we construct the perturbation $w_{q+1}$ by three parts: $\wpq$, $\wcq$ and $\wtq$. The  main perturbation $\wpq$ has compact support, and is  composed by the shear intermittent flows and the temporal concentration function. $\wcq$ is the incompressibility corrector and  $\wtq$ is the so-called  incompressible perturbation fluid.  In the new iterative scheme,  we decompose $u_q$ into two parts: $u^{\text{loc}}_q$ with compact support and $u^{\text{non-loc}}_q$  without compact support. Then the new Reynolds stress $\RR_{q+1}$ is
\begin{align*}
   & \Div \RR_{q+1}\sim \,  \Div \RR_q+\Div E_{q+1}+F_{q+1}\\
   &+\partial_t \wtq-\Delta \wtq+\Div(\wtq\otimes \wtq)+\Div(\wtq\otimes u^{\text{non-loc}}_q)+\Div( u^{\text{non-loc}}_q\otimes \wtq)+\nabla p_t .
\end{align*}
Here $E_{q+1}$ consists of the parts stemming from $\wpq$, $\wcq$ or $u^{\text{loc}}_q$, hence $E_{q+1}$ has compact support. The low frequency part of  $\wpq\otimes\wpq$, one part of $E_{q+1}$, cancels the Reynolds stress $\RR_{q}$ so that the size of the stress error $\Div(\RR_q+E_{q+1})$ can be reduced. For $F_{q+1}$  which corresponds to the parts of the non-divergence form derived from  $\wpq$, $\wcq$ or $u_q$, benefiting from the special structure of the shear intermittent flows such that the oscillation direction is perpendicular to the  direction of flow, one could expect that $F_{q+1}$ is small in a suitable sense. This in turn guarantees the existence of incompressible Navier-Stokes fluid $\wtq$ that  starts from the identically zero flow, is small in some space and cancels $F_{q+1}$. Consequently, part of $E_{q+1}$ constitutes the new Reynolds tensor $\RR_{q+1}$, which maintains the divergence form and possesses compact support. Then $u^{\text{loc}}_{q+1}\sim u^{\text{loc}}_{q}+\wpq+\wcq$ with compact support, and $u^{\text{non-loc}}_{q+1}\sim u^{\text{non-loc}}_{q}+\wtq$ is small.

To show the strong non-uniqueness, we construct weak solution $v$  that satisfies the  prescribed $L^2([\tfrac{3}{4}T, T], L^2(\R^3))$-norm
\begin{align*}
\int_{\frac{3}{4}T}^T\int_{\R^3}|v(x,t)|^2\dd x\dd t=E
\end{align*}
 with the same initial data. In fact, the  proof   consists of three steps. The
first step is mollifying the approximation solution $u_q$ as $u_{\ell_q}$  for establishing higher regularity estimates of the perturbation.  To avoid the mollification procedure interfering in  initial data, we introduce $\vv_q$ by  gluing $u_q$ and $u_{\ell_q}$ in the second step. In the third step,  we introduce incompressible perturbation  fluid to construct the perturbation $w_{q+1}$ of $\vv_q$ so that the iteration proceeds successfully in our iterative scheme. By developing the new iterative scheme, we firstly show non-uniqueness of weak solutions with non-compact space support for  the  Navier-Stokes equations through convex integration.

\section{Preliminaries}
{In this section, we compile several useful tools including geometric Lemma, an improved H\"{o}lder inequality and the definition of Lerner-Chemin spaces.
\begin{lemma}[Geometric Lemma \cite{2Beekie}]\label{first S}Let $B_{\sigma}({\rm Id})$ denote the ball of radius $\sigma$ centered at $\rm Id$ in the space of $3\times3$ symmetric matrices.
There exists a set $\Lambda\subset\mathbb{S}^2\cap\mathbb{Q}^3$ that consists of vectors $k$ with associated orthonormal basis $(k,\bar{k},\bar{\bar{k}}),~\epsilon>0$ and smooth function $a_{k}:B_{\epsilon}(\rm Id)\rightarrow\mathbb{R}$ such that, for every positive definite symmetric matrix $R\in B_{\epsilon}(\rm Id)$, we have the following identity:
$$R=\sum_{k\in\Lambda}a^2_{k}(R)\bar{k}\otimes\bar{k}.$$
\end{lemma}
\begin{remark}
  For instance, $\Lambda=\{\frac{5}{13}e_1\pm \frac{12}{13}e_2, \frac{12}{13}e_1\pm \frac{5}{13}e_3, \frac{5}{13}e_2\pm \frac{12}{13}e_3\}$ and $(k,\bar{k},\bar{\bar{k}})$ are as follows:
\begin{table}[ht]
\renewcommand\arraystretch{1.2}
\begin{tabular}{p{2cm}|p{2cm}|p{2cm}}
\toprule[1.2pt]
$\qquad k$& $\qquad {\bar{k}}$ & $\qquad {\bar{\bar{k}}}$\\\midrule[0.5pt]
$\frac{5}{13}e_1\pm \frac{12}{13}e_2$ &{$\frac{5}{13}e_1\mp\frac{12}{13}e_2 $}&\qquad  $e_3$ \\\hline
$\frac{12}{13}e_1\pm \frac{5}{13}e_3$&{$\frac{12}{13}e_1\mp \frac{5}{13}e_3$}& \qquad $e_2$\\\hline
$\frac{5}{13}e_2\pm \frac{12}{13}e_3$ &{$\frac{5}{13}e_2\mp \frac{12}{13}e_3$}& \qquad $e_1$\\\bottomrule[1.2pt]
\end{tabular}
\end{table}
\end{remark}

{We provide an improved H\"{o}lder's inequality. In fact, the improved H\"{o}lder's inequality on periodic functions is established in \cite{1Cheskidov, MoS}. From its proof, one easily  deduces the following improved H\"{o}lder's inequality for  functions with compact support.
\begin{lemma}[\cite{1Cheskidov, MoS}]\label{Holder}Assume that $d\ge1$, $1\le p\le \infty$, $\lambda$ and $L$ are positive integers. Let $\Omega=\big[-\tfrac{L}{2}, \tfrac{L}{2}\big]^d\subseteq\R^d$ and smooth function $f$ support on $\Omega$. $g: \TTT^d\to\R$ is a smooth function and $\TTT^d=\R^d/(L\ZZ)^d$. Then we have
\[\Big|\|fg(\lambda\cdot)\|_{L^p(\Omega)}-\|f\|_{L^p(\Omega)}\|g\|_{L^p(\Omega)}\Big|\lesssim L^{\frac{1}{p}}\lambda^{-\frac{1}{p}}\|f\|_{C^1(\Omega)}\|g\|_{L^p(\TTT^d)}.\]
 \end{lemma}}
\begin{proposition}\label{weak-glue-weak}For  $0<T_1<T_2$, let $u_1$ and $u_2$ be weak solutions of the equations \eqref{NS} respectively on $[0, T_1]$ and $[T_1, T_2]$ in a sense of Definition \ref{def} with $u_1(0)\in L^2(\R^3)$ and $u_1(T_1)=u_2(T_1)\in L^2(\R^3)$.  Then $u$  defined by
\begin{align*}
    u(t)=u_1(t), \quad \text{for}\,\,t\in[0, T_1], \quad\quad  u(t)=u_2(t), \quad \text{for}\,\,t\in[T_1, T_2]
\end{align*}
is a weak solution of \eqref{NS} in  Definition \ref{def} on  $[0, T_2]$.
\end{proposition}
\begin{proof}
     Let $\eta(t)\in C^{\infty}(\R)$ such that $\eta (t)\equiv1$ for $t\le -1$ and $\eta(t)\equiv 0$ for $t\ge 0$. Suppose that $\epsilon<T_1$,  we define $\eta^{(1)}_{\epsilon}(t)=\eta(\frac{t-T_1}{\epsilon})$ and $\eta^{(2)}_{\epsilon}(t)=\eta^{(1)}_{\epsilon}(-t)$. For any  $\phi\in \mathcal{D}_T$, since $u_1$ is a weak solution of \eqref{NS} on $[0, T_1]$ and $u\equiv u_1~{\rm on}~[0, T_1]$, we have
 \begin{align}\nonumber
\int_0^{T_1} \int_{\R^3} \big((\del_t-\Delta)\,( \eta^{(1)}_{\epsilon}(t) \phi )\big) u+\nabla (\eta^{(1)}_{\epsilon}(t)\phi): u\otimes u \dd x \dd t=  -\int_{\R^3} u_1(0) \phi(0,x)\dd x.
\end{align}
Noting the support of $\eta^{(1)}_{\epsilon}(t)$, the integration region in the above equality can be extended to $[0, T_2]$. Similarly, since  $u_2$ is a weak solution of \eqref{NS} on $[T_1, T_2]$ and $u\equiv u_2~{\rm on}~[T_1, T_2]$, one obtains that
 \begin{align}\nonumber
\int_{T_1}^{T_2} \int_{\R^3} \big((\del_t-\Delta)\,( \eta^{(2)}_{\epsilon}(t) \phi ) \big)u+\nabla (\eta^{(2)}_{\epsilon}(t)\phi) : u\otimes u \dd x \dd t=0,
\end{align}
where the  integration region can be extended to $[0, T_2]$. Collecting these equalities together shows that
\begin{align}\nonumber
&\int_{0}^{T_2} \int_{\R^3} \big((\del_t-\Delta)\, ((\eta^{(1)}_{\epsilon}(t)+\eta^{(2)}_{\epsilon}(t)) \phi) \big) u+\nabla ((\eta^{(1)}_{\epsilon}(t)+\eta^{(2)}_{\epsilon}(t))\phi) : u\otimes u \dd x \dd t\nonumber\\
=&-\int_{\R^3} u_1(0) \phi(0,x)\dd x.\nonumber
\end{align}
A direct computation yields that
\begin{align*}
&\,\,\quad\int_{0}^{T_2}  \int_{\R^3} \big((\del_t-\Delta)\, ((\eta^{(1)}_{\epsilon}(t)+\eta^{(2)}_{\epsilon}(t)) \phi) \big) u\dd x \dd t\\
&=\int_{0}^{T_2} \int_{\R^3}  ((\del_t-\Delta)\,\phi ) u (\eta^{(1)}_{\epsilon}(t)+\eta^{(2)}_{\epsilon}(t)) \dd x \dd t+\int_{0}^{T_2}  \int_{\R^3}  \phi  u \,\del_t (\eta^{(1)}_{\epsilon}(t)+\eta^{(2)}_{\epsilon}(t)) \dd x \dd t\\
&=:{\rm I+II}.
\end{align*}
For {\rm I}, owning to $u\in L^2([0,T_2]\times \R^3)$ and
\begin{align*}
    \eta^{(1)}_{\epsilon}(t)+\eta^{(2)}_{\epsilon}(t)\to 1\,\,\text{in}\,\, L^2, \,\,\text{as}\,\,\epsilon\to 0,
\end{align*}
one deduces that
\begin{align*}
{\rm I}\to\int_{0}^T \int_{\mathbb T^3} (\del_t-\Delta)\phi\, u\dd x \dd t, \quad \text{as}\quad\epsilon \to 0.
\end{align*}
For {\rm II}, by the definitions of $\eta^{(1)}_{\epsilon}(t)$ and $\eta^{(2)}_{\epsilon}(t)$, we have
\begin{align*}
{\rm II}=\frac{1}{\epsilon}\int_{T_1-\epsilon}^{T_1}  \eta'\Big(\frac{t-T_1}{\epsilon}\Big)\int_{\R^3} u_1\phi\dd x \dd t-\frac{1}{\epsilon}\int_{T_1}^{T_1+\epsilon} \eta'\Big(\frac{t-T_1}{\epsilon}\Big)\int_{\R^3} u_2\phi\dd x \dd t.
\end{align*}
By \cite[Theorem 2.1]{FJR},  the  definition of weak solutions in Definition \ref{def} is equivalent to the integral equation, from which we infer that $\int_{\R^3}u_1\phi \dd x\in C([0,T_1])$ and $\int_{\R^3}u_2\phi \dd x\in C([T_1,T_2])$. Therefore, $T_1$  is a Lebesgue point for $\int_{\R^3}u_1\phi \dd x$ and $\int_{\R^3}u_2\phi \dd x$.  By Lebesgue differentiation theorem, we have
\begin{align*}
{\rm II}\to \eta'(0) \int_{\R^3} u_1(T_1)\phi(T_1)\dd x \dd t- \eta'(0)\int_{\R^3}u_2(T_1)\phi(T_1)\dd x \dd t=0, \quad \text{as}\quad \epsilon\to 0.
\end{align*}
Thanks to
\begin{align*}
\eta^{(1)}_{\epsilon}(t)+\eta^{(2)}_{\epsilon}(t)\to 1, \forall t\in \R, \quad \text{and}  \quad  u\otimes u\in L^1([0,T]\times\R^3),
\end{align*}
one deduces by dominated convergence theorem that
\begin{align*}
&\int_{0}^{T_2} \int_{\mathbb T^3}\nabla \big((\eta^{(1)}_{\epsilon}(t)+\eta^{(2)}_{\epsilon}(t))\phi\big) : u\otimes u\dd x \dd t\\
=&\int_{0}^{T_2}  \int_{\mathbb T^3}(\eta^{(1)}_{\epsilon}(t)+\eta^{(2)}_{\epsilon}(t))\nabla \phi : u\otimes u\dd x \dd t
\to \int_{0}^T \int_{\mathbb T^3}\nabla \phi : u\otimes u \dd x \dd t, \,\, \text{as}\quad\epsilon \to 0.
\end{align*}
Hence, we obtain that
\begin{align*}
&\int_{0}^{T_2} \int_{\R^3} \big((\del_t-\Delta)\,( (\eta^{(1)}_{\epsilon}(t)+\eta^{(2)}_{\epsilon}(t)) \phi )\big) u+\nabla ((\eta^{(1)}_{\epsilon}(t)+\eta^{(2)}_{\epsilon}(t))\phi) : u\otimes u\dd x \dd t\\
\to &\int_0^{T_2}\int_{\R^3} ((\del_t-\Delta)\, \phi ) u+\nabla \phi : u\otimes u \dd x \dd t\quad \text{as}\quad \epsilon\to 0.
\end{align*}
In conclusion, one  has
\begin{align*}
\int_0^{T_2}\int_{\R^3} ((\del_t-\Delta)\, \phi ) u+\nabla \phi : u\otimes u\dd x \dd t=-\int_{\R^3} u_1(0) \phi(0,x)\dd x.
\end{align*}
\end{proof}

{Next, we give the temporal and spatial mollifiers  which we will use in Section 4.
\begin{definition}[Mollifiers]\label{e:defn-mollifier-t}Let nonnegative functions $\varphi(t)\in C^\infty_c(-1,0)$ and $\psi(x)\in C^\infty_c(B_1(0))$ be standard mollifying kernels such that $\int_{\R}\varphi(t)\dd t=\int_{\R^3}\psi(x)\dd x=1$.
For each $\epsilon>0$, we define  two sequences of  mollifiers as follows:
\begin{align*}
     \varphi_{\epsilon}(t)
        := \frac1{\epsilon} \varphi\left(\frac t\epsilon\right), \quad\psi_\epsilon(x)
            :=  \frac1{\epsilon^3} \psi\left(\frac{x}\epsilon\right).
\end{align*}
\end{definition}
In this paper, we will introduce the incompressible perturbation  fluid  in the following mixed time-spatial Besov spaces, the so-called Lerner-Chemin spaces.
\begin{definition}[\cite{BCD11, MWZ}]Let $T>0$, $s\in\R$ and $1\le r,p,q\le\infty$. The mixed  time-spatial Besov spaces ${L}^r_TB^s_{p,q}$ consists of all $u\in\mathcal{S}'$ satisfying
\begin{align*}
\|u\|_{\widetilde{L}^r_TB^s_{p,q}(\R^d)}\overset{\text{def}}{=}\Big{\|}(2^{js}\|\Delta_j u\|_{L^r([0,T];L^p(\R^d))})_{j\in\ZZ}\Big{\|}_{\ell^q(\ZZ)}<\infty,
\end{align*}
where $\Delta_j$ is localization nonhomogeneous operator from the Littlewood-Paley decomposition theory. Particularly, $H^s(\R^d)\sim B^s_{2,2}(\R^d)$.
\end{definition}
We present a result describing  the smoothing effect of the heat flow  in the context of Besov spaces.
\begin{lemma}[\cite{BCD11, MWZ}]\label{heat}Let $s\in \R$ and $1\le r_1, r_2, p, q\le \infty$ with $r_2\le r_1$. Consider the heat equation
\begin{align*}
\partial_t u-\Delta u=f,\qquad
  u(0,x)=u_0(x).
  \end{align*}
Assume that $u_0\in  B^s_{p,q}(\R^d)$ and $f\in \widetilde {L}^{r_2}_t( B^{s-2+\frac{2}{r_2}}_{p,q}(\R^d))$. Then the above equation has a unique solution $u\in \widetilde{L}^{r_1}_t( \dot B^{s+\frac{2}{r_1}}_{p,q}(\R^d))$
satisfying
\[\|u\|_{\widetilde{L}^{r_1}_T(  B^{s+\frac{2}{r_1}}_{p,q}(\R^d))}\le C(1+T)\big(\|u_0\|_{ B^s_{p,q}(\R^d)}+\|f\|_{\widetilde{L}^{r_2}_T( \dot B^{s-2+\frac{2}{r_2}}_{p,q}(\R^d))}\big),\]
where $C$ is a universal constant.
\end{lemma}
\noindent {\bf{Notation}}\, For a $\TTT^d$-periodic function $f$, we denote
\begin{align*}
\mathbb{P}_{=0} f:=\widehat{f}(0)=\frac{1}{|\TTT^d|}\int_{\TTT^d} f(x)\dd x,\quad\PP f=f-\mathbb{P}_{=0} f.
\end{align*}
In the following, the notation $x\lesssim y$ means $x\le Cy$ for a universal constant that may change from line to line. We use the symbol $\lesssim_N$ to express that the constant in the inequality depends on the parameter $N$. Without ambiguity, we will denote $L^m([0,T];Y(\R^3))$ and $L^m([0,T];L^m(\R^3))$ by $L^m_t Y$ and $L^m_{t,x}$ respectively. }
\section{ Induction scheme}
\subsection{Parameters}\label{para}First of all, we introduce several parameters throughout this paper. Fixed $0<T<1$ with $T\in\QQ$, $ N_{\Lambda}$ and $M$ be positive integer and  $1\le p<2$,  let $K$ be integer number with $K>\max\{\frac{4}{T},2M\}$. We define $\epsilon_0,\sigma_0\in\QQ$ and $\alpha$ be positive constants
\begin{align}\label{epsilon}
2\epsilon_0\le \min\big\{2^{-12}, \tfrac{2}{p}-1\big\},\quad \sigma_0\le \frac{\epsilon_0}{20}, \quad  \alpha<\frac{\sigma_0}{100}.
\end{align}
Let $b\in \NN$ with $b(1-\epsilon_0)\in \NN$,  $b\sigma_0\in \NN$ and
\begin{align}\label{b-beta}
b\ge\frac{2^{20}}{\alpha}, \quad \beta=2^{-20}\alpha b^{-1}.
\end{align}
Suppose that $a\in \NN$ satisfying that $\frac{a}{T}\in 4\NN$ and $a>\max\{K,N_{\Lambda}\}$. We define
\begin{align}\label{def-lq}
    \lambda_q :=    a^{b^q},  \quad \delta_q :=  \lambda_q^{-2\beta}, \quad \ell_q:=\lambda^{-50}_q,\quad q\ge 0,
\end{align}
and
\begin{align}\label{omega}
\Omega_{q}:=\big[-\tfrac{K}{2} +(\lambda_{q-1}\delta^{1/2}_{q-1})^{-1},~
 \tfrac{K}{2} -(\lambda_{q-1}\delta^{1/2}_{q-1})^{-1}  \big]^3   ,~~~~q\geq1.
\end{align}
%\subsection{The Navier-Stokes-Reynolds system}Let us introduce the  relaxation of the system \eqref{NS} with a stress tensor error term that  tends to $0$ in the sense of distributions. More precisely, the approximate system, the so-called Navier–Stokes-Reynolds  system is governed by
%\begin{equation}\label{NSR}\tag{NSR}
%\left\{ \begin{alignedat}{-1}
%   & \del_t u_q-\Delta u_q+\Div (u_q\otimes u_q)  +\nabla p_q   =\RR_q,
% \\
 % &\nabla \cdot u_q = 0, \\
%\end{alignedat}\right.
%\end{equation}
%associated with some initial data $\uin\in H^3(\R^3)$.
%The \emph{Reynolds stress} $\RR_q$ is trace-free symmetric matrix. Here and below,  $v \otimes u\coloneq  (v_i u_j)_{i,j=1}^3$, and the divergence of a $3\times 3$ matrix $M=(M_{ij})_{i,j=1}^3$ is defined by $\Div M$ with components $(\Div M)_j \coloneq   \partial_i M_{ij}$.
%We enforce the conditions on $u_q$ and $p_q$:
%\[\int_{\R^3}u_q\dd x=\int_{\R^3}p_q\dd x=0.\]
\subsection{Iterative procedure }\label{sec-ite}For given initial data $\uin\in H^3(\R^3)$, there exist $0<T<1$ and a smooth solution $u_1\in C([0,T]; H^3(\R^3))\cap L^2([0,T]; H^4(\R^3))$ of the equations \eqref{NS}.  We choose the constant $E$ such that
\begin{align}\label{E-u1}
 \int_{\frac{3}{4}T}^T\int_{\R^3}|u_1(x,t)|^2\dd x\dd t+2\delta_2\le E\le \int_{\frac{3}{4}T}^T\int_{\R^3}|u_1(x,t)|^2\dd x\dd t+4\delta_2.
\end{align}
Furthermore, for given positive constant $K$, we define the spatial cut-off function $\chi_K\in C^\infty(\R^3; [0,1])$ such that
\begin{align*}
\chi_K(x)=1, \text{if}\,\,|x|\le \frac{K}{8} \,\,\text{and}\,\,  \chi(x)=0, \text{if}\,\,|x|\ge \frac{K}{4}.
\end{align*}
Then we decompose $u_1$ by
\begin{align*}
    u_1=u_1 \chi_K+u_1(1-\chi_K):=u^{\loc}_1+u^{\text{non-loc}}_1.
\end{align*}
Owing to $u_1\in C([0, T]; H^3(\R^3))$,  we have
\begin{align}\label{u1-L^2}
   \|u_1\|_{L^2([0,T];L^2(\R^3)))\cap L^p([0,T];L^\infty(\R^3)))}\le \frac{M}{4}
\end{align}
for a large enough constant $M$ and there exists a large enough integer $K$ such that
\begin{align}\label{u1nl}
\|u^{\text{non-loc}}_1\|_{\widetilde{L}^{\infty}([0,T];B^{1/2}_{2,1}(\R^3))\cap \widetilde{L}^{1}([0,T];B^{5/2}_{2,1}(\R^3)) }\le \frac{1}{2M},
\end{align}

To employ induction, we suppose that the solution $(u_q, p_q, \RR_q)$ of the equations \eqref{NSR} on $[0, T]\times \R^3$ satisfies the following conditions:
\begin{align}
&\uq=\uql+\uqnl,
    \label{uq-tigh}\\
 &\|\uq\|_{L^2_{t,x}}\leq M(1- \delta^{1/2}_q),  \quad\,\,\|\uq\|_{L^\infty_tH^3}  \le  \lambda^5_q,
    \label{e:vq-H3}\\
&\|\uql\|_{L^2_{t,x}\cap L^p_tL^{\infty}}\leq\frac{M}{2}(1-\delta^{1/2}_q), \quad {\text{supp}_{x}}   \uql \subseteq \Omega_q   ,\label{e:vq-C0}\\
    &\|\uqnl\|_{\widetilde{L}^{\infty}_tB^{3}_{2,2}}  \le  \lambda^5_q ,\,\quad\|\uqnl\|_{\widetilde{L}^{\infty}_tB^{1/2}_{2,1}\cap \widetilde{L}^{1}_tB^{5/2}_{2,1} }\le M^{-1}+\sum_{k=2}^{q} \delta_{k+1}\lambda^{-6\alpha}_k,
    \label{e:vqnl-H3}\\
 & 2\delta_{q+1}\le E-\int_{\frac{3}{4}T}^T\int_{\R^3}|\uq|^2\dd x\dd t\le 4\delta_{q+1}, \label{e:E-q}\\
&   {\text{supp}_{x}} \RR_q  \subseteq \Omega_q, \qquad\qquad\quad\,\,\RR_q(t,x)=0, \forall\, t\in [0, \tfrac{T}{4}+4\lambda^{-1}_{q-1}],\label{supp-Rq}\\
    & \| \RR_q \|_{L^1_{t,x}}\le \delta_{q+1}\lambda_q^{-4 \alpha},  \quad\quad\quad\| \RR_q \|_{L^{\infty}_{t}W^{3,1}} \le  \lambda_q^{5},
    \label{e:RR_q-C0}
\end{align}
where $M$ and $E$ are consistent with these in \eqref{E-u1}--\eqref{u1nl}.   The following proposition shows that there exists a solution  $(u_{q+1}, p_{q+1}, \RR_{q+1})$ of the equations~\eqref{NSR}
satisfying the above inductive conditions \eqref{uq-tigh}--\eqref{e:RR_q-C0} with $q$ replaced by $q+1$, which guarantees the iteration proceeds successfully. Indeed,  by employing the iterative proposition as presented below, we are able to prove Theorem ~\ref{t:main0}.
{\begin{proposition}\label{iteration}
\label{p:main-prop}Let $1\le p<2$ and the parameters $M,T,\lambda_q,\delta_q,\Omega_q$ be as in \eqref{def-lq}--\eqref{u1nl}. Then there exist  a  universal constant $C_0$ and $a_0$ such that for $a>a_0$, the following holds.
Assume that $(u_q,p_q,\RR_q )$ solves
\eqref{NSR} with $q\ge1$ . Then there exists a solution $ (u_{q+1},  p_{q+1}, \RR_{q+1})$ of the equations \eqref{NSR} on $[0,T]$, satisfying \eqref{uq-tigh}--\eqref{e:RR_q-C0} with $q$ replaced by $q+1$, and such that
\begin{align}
&u_{q+1}(t)=u_q(t), \quad \,\forall t\in [0, \tfrac{T}{4}+4\lambda^{-1}_{q}], \label{uq+1=uq}\\
& \|u_{q+1} - u_q\|_{L^2_{t,x}\cap L^p_tL^{\infty}_x} \leq C_0\delta_{q+1}^{1/2}.\label{uq+1-uq}
\end{align}
\end{proposition}}
\subsection{Proof of Theorem \ref{t:main0}} Given a weak solution $u\in L^p([0, T_0]; L^\infty(\R^3))$ with initial data $u_0$,
there exists $t_0\in (0, T_0)$ such that $u(t_0)\in L^2(\R^3)\cap L^\infty(\R^3)$. By the well-posedness theory of the Navier-Stokes equations, we have a unique local mild solution $u_1$ on $[t_0,t_{\rm local}]\subseteq[0, T_0]$ with initial data $u_1(t_0)=u(t_0)$ and there exists $0<4\varepsilon<t_{\rm local}-t_0$ such that
$$u_1(t_0+\varepsilon)\in H^3(\R^3).$$  Let $u_1(t_0+\varepsilon)\in H^3(\R^3)$ be  the initial data $\uin$ as given in Section \ref{sec-ite}, then $u_1$ is a smooth solution of the equations \eqref{NS} on $[t_0+\varepsilon, t_0+\varepsilon+T]$ with some $T\le \min\{1, \frac{1}{2}(t_{\rm local}-t_0)\}$. Obviously, by taking $a$ large enough and combining \eqref{E-u1}--\eqref{u1nl}, $u_1$ satisfies \eqref{uq-tigh}--\eqref{e:E-q} and $\RR_1$ satisfies \eqref{supp-Rq} and \eqref{e:RR_q-C0} due $\RR_1=0$ at $q=1$.

Utilizing Proposition \ref{p:main-prop} inductively, we obtain a sequence of solutions $\{(u_q, p_q,\RR_q)\}$ of the equations \eqref{NSR} satisfying the inductive estimates \eqref{uq-tigh}--\eqref{e:RR_q-C0}. By the definition of $\delta_q$, one can easily deduce that $\sum_{i=2}^{\infty}\delta^{1/2}_{i}$ converges to a finite number. This fact combined with \eqref{uq+1-uq}
    implies that $\{u_q\}$ is a Cauchy sequence in $L^2([t_0+\varepsilon, t_0+\varepsilon+T]; L^2(\R^3))\cap L^p([t_0+\varepsilon, t_0+\varepsilon+T]; L^\infty(\R^3))$. Since  $\|\RR_q\|_{L^1_{t,x}}\rightarrow 0$ as $q\to\infty$,  the limit function $\widetilde v$  is a weak solution of \eqref{NS}  and satisfies
    \begin{align}
    &\widetilde v\in L^2([t_0+\varepsilon, t_0+\varepsilon+T]; L^2(\R^3))\cap L^p([t_0+\varepsilon, t_0+\varepsilon+T]; L^\infty(\R^3)),\label{weak-v-L^2}\\
    &\int_{t_0+\varepsilon+\frac{3}{4}T}^{t_0+\varepsilon+T}\int_{\R^3}|\widetilde v|^2\dd x\dd t=E.
    \end{align}
 Thanks to \eqref{uq+1=uq}, we have $\widetilde v(t_0+\varepsilon)=u_1(t_0+\varepsilon)$.

By virtue of \eqref{weak-v-L^2}, one infers that $\widetilde v(x, t)\in L^2(\R^3)$ for a.e. $t\in [t_0+\varepsilon, t_0+\varepsilon+T]$. Without loss of generality, we assume that $\widetilde v(x, t_0+\varepsilon+T)\in L^2(\R^3)$. Then there exists a Leray-Hopf solution $V$ of \eqref{NS} on $[t_0+\varepsilon+T, \infty)$, which is also a weak solution of \eqref{NS}  on $[t_0+\varepsilon+T, \infty)$ with initial data $\widetilde v(t_0+\varepsilon+T)$.

Finally, we construct a weak solution $v$ of \eqref{NS} on $[0, T_0]$ by letting
\begin{align*}
    &v(t)= u(t) \,\, \text{if}\,t\in[0,t_0]; \quad\quad\quad\quad\quad\quad v(t)= u_1(t) \,\, \text{if}\,t\in[t_0,t_0+\varepsilon];\\
    &v(t)=\widetilde v(t) \,\, \text{if}\,t\in [t_0+\varepsilon, t_0+\varepsilon+T]; \, \,\,v(t)=V(t)\,\, \text{if}\,t\in  [ t_0+\varepsilon+T, T_0].
\end{align*}
Note that $u$, $u_1$, $\widetilde{v}$ and $V$ are weak solutions of the  \eqref{NS} with $u(t_0)=\widetilde u(t_0)$, $\widetilde u(t_0+\varepsilon)=\widetilde v(t_0+\varepsilon)$ and $\widetilde v(t_0+\varepsilon+T)=V(t_0+\varepsilon+T)$, one could deduce that $v$ is a weak solutions of the equations~\eqref{NS} on $[0, T_0]$ by Proposition \ref{weak-glue-weak}.
Since the weak solution $v$ satisfies
\[\int_{t_0+\varepsilon+\frac{3}{4}T}^{t_0+\varepsilon+T}\int_{\R^3}| v(x,t)|^2\dd x\dd t=E,\]
where there exist infinitely many constant $E$ satisfying \eqref{E-u1},  we conclude that  $v\neq u$. Therefore, we finish the proof of Theorem \ref{t:main0}.
\section{Proof of Proposition \ref{iteration}} In this section, we are devoted to the proof of Proposition \ref{iteration}. More specifically, we construct  $(u_{q+1}, p_{q+1}, \RR_{q+1})$ in Proposition \ref{p:main-prop}  by the following three steps:
\begin{enumerate}
  \item [$\bullet$]Step 1: Mollification: $(\uq, p_q, \RR_q)\mapsto (u_{\ell_q},  p_{\ell_q}, \RR_{\ell_q})$. We define $(u_{\ell_q},  p_{\ell_q}, \RR_{\ell_q})$ by mollifying  $(\uq, p_q, \RR_q)$ so that  we  {obtain higher regularity estimates for the perturbation.}
  \item [$\bullet$]Step 2: Gluing:  $(u_{\ell_q},  p_{\ell_q}, \RR_{\ell_q})\mapsto (\vv_q, \pp_q, \RRR_q)$. We introduce $\vv_q$ by gluing $\uq$ and  $u_{\ell_q}$, which is employed to achieve the condition \eqref{uq+1=uq}.
 \item [$\bullet$]Step 3: Perturbation: $ (\vv_q,  \pp_q, \RRR_q)\mapsto (u_{q+1}, p_{q+1}, \RR_{q+1})$. By making use of the shear intermittent flows, the temporal concentration functions and incompressible  perturbation fluid, we construct the perturbation $w_{q+1}$. Then we define $u_{q+1}$ by adding $w_{q+1}$ on~$\vv_q$.
\end{enumerate}
\subsection{Mollification}
We define the functions $(u_{\ell_q},  p_{\ell_q}, \RR_{\ell_q})$ by the spatial  mollifier $\psi_{\ell_q}$ and the time mollifier $\varphi_{\ell_q}$ in Definition \ref{e:defn-mollifier-t} as follows: For $(x,t)\in \R^3\times [0, T]$,
\begin{align*}
   & u_{\ell_q} (x,t):=\int_t^{t+\ell_q}(\uq * \psi_{\ell_q})(x,s)\varphi_{\ell_q}(t-s)\dd s, \\
   &\ulql(x,t):=\int_t^{t+\ell_q}(\uql * \psi_{\ell_q})(x,s)\varphi_{\ell_q}(t-s)\dd s, \\
      &\ulqnl(x,t):=\int_t^{t+\ell_q}(\uqnl * \psi_{\ell_q})(x,s)\varphi_{\ell_q}(t-s)\dd s, \\
   &p_{\ell_q}(x,t):=\int_t^{t+\ell_q}(p_q * \psi_{\ell_q})(x,s)\varphi_{\ell_q}(t-s)\dd s,\\
 &   \RR_{\ell_q}(x,t) := \int_t^{t+\ell_q}(\RR_q * \psi_{\ell_q})(x,s)\varphi_{\ell_q}(t-s)\dd s, \\
 &\Rem:=  u_{\ell_q} \otimes u_{\ell_q} - \int_t^{t+\ell_q}((u_q \otimes u_q) * \psi_{\ell_q})(x,s)\varphi_{\ell_q}(t-s)\dd s.
\end{align*}
 One easily verifies that  $(u_{\ell_q},  p_{\ell_q}, \RR_{\ell_q}, \Rem)$ solves the Cauchy problem
\begin{equation}
\left\{ \begin{alignedat}{-1}
&\del_t u_{\ell_q}-\Delta u_{\ell_q}+\Div (u_{\ell_q}\otimes u_{\ell_q})  +\nabla p_{\ell_q}   =  \Div \RR_{\ell_q} +\Div \Rem,
\\
 & \nabla \cdot u_{\ell_q} = 0,
  \\
  &  u_{\ell_q} |_{t=0}=   \int_0^{\ell_q}(u_q * \psi_{\ell_q})(x,s)\varphi_{\ell_q}(-s)\dd s
\end{alignedat}\right.
 \label{e:mollified-euler}
\end{equation}
{such that
\begin{align}
&\supp_x \ulql\subseteq \Omega_{q}+[-\lambda^{-1}_q, \lambda^{-1}_q]^3,\label{supp-vlq}\\
&\RR_{\ell_q}(t,x)=0, \quad\forall\, t\in [0, \tfrac{T}{4}+3\lambda^{-1}_{q-1}],\quad \supp_x \RR_{\ell_q} \subseteq \Omega_{q}+[-\lambda^{-1}_q, \lambda^{-1}_q]^3.\label{supp-Rlq}
\end{align}
Moreover,  $(u_{\ell_q}, \RR_{\ell_q}, \Rem)$ satisfies the following estimates.
\begin{proposition}[Estimates  for $(u_{\ell_q}, \RR_{\ell_q}, \Rem)$]\label{p:estimates-for-mollified}For any integers $L,N\ge 0$, we have
\begin{align}
&\|\partial^L_t u_{\ell_q} \|_{L^\infty_tH^{N+3}} \lesssim  \lambda^5_{q} \ell_q^{-N-L},  \label{e:v_ell-CN+1}
\\
&{\|\RR_{\ell_q}\|_{L^1_{t,x}} \le  \delta_{q+1} \lambda^{-4\alpha}_q}, && \label{e:R_ell}
\\
&\|\partial^L_t\RR_{\ell_q}\|_{L^\infty_tW^{N+3,1}} \lesssim   \lambda^5_{q} \ell_q^{-N-L},\label{e:R_ell-W} \\
%&\|\RR_{\ell_q}\|_{C^{N+1}_{t,x}} \lesssim   \lambda^5_{q} \ell_q^{-N}, && \forall N\geq 0 .\label{e:t_R_ell-C} \\
&\|\RR^{rem}_q\|_{L^\infty_tH^{N+2}} \lesssim  \lambda^{10}_{q} \ell_q^{-N+1}.\label{e:R_rem}
\end{align}
\end{proposition}}
\subsection{Gluing procedure}
Let $\zeta_q(t)\in C^\infty(\R)$ be defined by
\begin{equation}\label{def-zeta}
    \zeta_q(t)=1, \,\, t\le \tfrac{T}{4}+2\lambda^{-1}_{q-1}; \,\,\zeta_q(t)=0, \,\, t\ge \tfrac{T}{4}+3\lambda^{-1}_{q-1}; \,\,|\partial^N_t\zeta_q|\le C\lambda^N_{q-1}.
\end{equation}
We define $\vv_q$ by
\begin{equation}\label{def-vvq}
\vv_q=\zeta_q u_q+(1-\zeta_q)u_{\ell_q},
\end{equation}
and
\[\ubql=\zeta_q \uql+(1-\zeta_q)\ulql,\quad \ubqnl=\zeta_q \uqnl+(1-\zeta_q)\ulqnl.\]
Thanks to \eqref{e:vq-H3}--\eqref{e:vqnl-H3}, one immediately shows that
\begin{align}
    &\|\ubql\|_{L^2_{t,x}\cap L^p_tL^{\infty}}\le \frac{M}{2}(1-\delta^{1/2}_q),\quad \supp_x \ubql\subseteq \Omega_{q}+[-\lambda^{-1}_q, \lambda^{-1}_q]^3;\label{e:ubql}\\
   & \|\ubqnl\|_{\widetilde{L}^{\infty}_tB^{3}_{2,2}}  \le  \lambda^5_q,\,\quad\|\ubqnl\|_{\widetilde{L}^{\infty}_tB^{1/2}_{2,1}\cap \widetilde{L}^{1}_tB^{5/2}_{2,1} }\le  M^{-1}+\sum_{k=2}^{q} \delta_{k+1}\lambda^{-6\alpha}_k;\label{e:ubqnl}\\
   &\|\vv_q\|_{L^2_{t,x}}\leq M(1- \delta^{1/2}_q),  \quad\,\,\|\vv_q\|_{L^\infty_tH^3}  \le  \lambda^5_q.\label{e:ubq}
\end{align}
Moreover, we infer from \eqref{e:vq-H3} and \eqref{def-vvq} that
\begin{align}\label{uq-ubq}
\|u_q-\vv_q\|_{L^2_{t,x}\cap L^p_tL^{\infty}}=\|(1-\zeta_q)(u_q-\ulql)\|_{L^2_{t,x}\cap L^p_tL^{\infty}}\le C\ell_q\|u_q\|_{L^\infty_tH^3}\le \lambda^{-40}_q.
\end{align}
This estimate together with \eqref{e:E-q} yields that
\begin{align}\label{e:E-vvq}
    \delta_{q+1}\le E-\int_{\frac{3}{4}T}^T\int_{\R^3}|\vv_q|^2\dd x\dd t\le 5\delta_{q+1}.
\end{align}
It follows from \eqref{supp-Rq} and \eqref{supp-Rlq} that
$$\zeta_q\, \Div \RR_q=0, \quad (1-\zeta_q)\,\Div \RR_{\ell_q}=\Div \RR_{\ell_q}.$$ Therefore, one obtains that
\begin{align}\label{ubq-t}
\vv_q(t)=u_q(t), \,\,t\in[0, \tfrac{T}{4}+2\lambda^{-1}_{q-1}]
\end{align}
 and  satisfies
\begin{equation}
\left\{ \begin{alignedat}{-1}
&\del_t \vv_q-\Delta \vv_q+\Div (\vv_q\otimes \vv_q)  +\nabla \pp_q =  \Div \RR_{\ell_q} +(1-\zeta_q)\Div \RR^{rem}_q+\partial_t\zeta_q(u_q-u_{\ell_q})\\
&\qquad\qquad\qquad\qquad\qquad\qquad\qquad\quad+\zeta_q(1-\zeta_q)\Div ((u_q-u_{\ell_q})\otimes (u_q-u_{\ell_q})),
\\
 & \nabla \cdot \vv_q = 0,
  \\
  &  \vv_q |_{t=0}= \uin.
\end{alignedat}\right.
 \label{e:mollified-euler}
\end{equation}

{\subsection{Perturbation} \textit{Intermittent shear velocity flow}\quad Assume that $\psi:\mathbb{R}\rightarrow\mathbb{R}$ is a smooth cutoff function supported on the interval $(0, \lambda^{-1}_1]$ and $\int_{\R} \psi''(x) \dd x=0$. We set $\phi=\frac{\dd^2}{\dd x^2}\psi$ and normalize it in such a way that
$$\int_{\mathbb{R}}\phi^2\dd x=1.$$
For any small positive parameter $r$ such that $r^{-1}\in\ZZ$ and $Q>0$, we define
\begin{align}\label{def-phir}
\phi_{r,Q}(x) :=Qr^{-\frac{1}{2}}\phi (r^{-1}x),\qquad \psi_{r,Q}(x) :=Qr^{-\frac{1}{2}}\psi (r^{-1}x).
\end{align}
Particularly, we denote
\begin{align*}
\phi_{r}(x) :=\phi_{r,K^{1/2}}(x), \,\,\psi_{r}(x) :=\psi_{r,K^{1/2}}(x)\,\,\text{and} \,\,\widetilde{\phi}_{r}(t):=\phi_{r, \frac{T^{1/2}}{4}}(t).
\end{align*}
 We periodize $\phi_r$ and $\psi_r$ so that  the resulting functions are periodic functions defined on $\mathbb{R}/ K\mathbb{Z}=:\mathbb{T}$, and we still denote the $\mathbb{T}$-periodic functions by $\phi_r$ and $\psi_r$. On the other hand, we periodize $\widetilde\phi_r$ so that  the resulting function is periodic function defined on  $\mathbb{R}/ {\tfrac{T}{4}}=:\widetilde{\mathbb{T}}$ and we denote the $\widetilde{\mathbb{T}}$-periodic functions by $\widetilde{\phi_r}$. }
\begin{proposition}[Estimates for the Fourier coefficient]\label{est-cmr}For $\phi_r(x)$, we rewrite it by Fourier series as follows:
\begin{align*}
&\phi_r(x)=\sum_{m\in\ZZ\backslash\{0\}} c_{m,r}e^{\ii  {2\pi}  mx /{K}}
\end{align*}
Then we have
\begin{equation}\label{cm-2}
|c_{m,r}|\le CK^{9/2}r^{-\frac{9}{2}}m^{-4}.
\end{equation}
\end{proposition}
\begin{proof}By the definitions of $c_{m,r}$ and $\phi_r$ , we have
    \begin{align*}
c_{m,r}=&\frac{1}{K}\int_{\TTT}\phi_r(x) e^{-\ii  2\pi  mx/K}\dd x=K^{-1/2}\int_{0}^{K}r^{-\frac{1}{2}}\phi(r^{-1}x)e^{-\ii  2\pi  mx/K}\dd x\\
=&K^{-1/2}\int_{0}^{r^{-1}K}r^{\frac{1}{2}}\phi(y) e^{- \ii 2\pi  mry/K}\dd y=K^{-1/2}r^{\frac{1}{2}}\int_{0}^{r^{-1}K}\phi(y) \tfrac{\ii K}{2\pi mr} \dd e^{- \ii 2\pi  mry/K}\\
=&-K^{-1/2}r^{\frac{1}{2}}\int_{0}^{r^{-1}K} \tfrac{\ii K}{2\pi mr}  e^{- \ii 2\pi  mry/K}  \phi'(y)\dd y.
\end{align*}
By integration by parts, one infers that
 \begin{align*}
c_{m,r}=& K^{-1/2}r^{\frac{1}{2}}\int_{0}^{r^{-1}K} \tfrac{ K^2}{(2\pi mr)^2}  \phi'(y)\dd e^{- \ii 2\pi  mry/K}\\
=&- K^{-1/2}r^{\frac{1}{2}}\int_{0}^{r^{-1}K} \tfrac{ K^2}{(2\pi mr)^2}  e^{- \ii 2\pi  mry/K} \phi''(y)\dd y.
\end{align*}
Hence, we show that
\begin{align*}
|c_{m,r}|\le CK^{5/2}r^{-5/2}m^{-2}.
\end{align*}
Using integration by parts twice again, we have
\begin{align*}
|c_{m,r}|\le CK^{9/2}r^{-9/2}m^{-4}.
\end{align*}
Hence, we complete the proof of Proposition \ref{est-cmr}.
\end{proof}
Next, we choose a large number  $N_{\Lambda}$ such that
$N_{\Lambda}{k}\in \mathbb{Z}^3$ for all $k\in\Lambda$. By \eqref{def-phir}, we introduce the following functions:
\begin{align*}
    \phi_{(\epsilon_0, 1-\epsilon_0, k)}(x): =&\phi_{\lambda^{-\epsilon_0}_{q+1}}(\lambda^{1-\epsilon_0}_{q+1}N_{\Lambda}k\cdot x),\quad
    g_{(2, \sigma_0, k)}(t):=\widetilde\phi_{\lambda^{-2}_{q+1}}({\lambda^{\sigma_0}_{q+1}}(t-t_k)),
\end{align*}
where we choose suitable shifts $t_k$ with $k\in\Lambda$ so that
$g_{(2, \sigma_0, k)}\cdot g_{(2, \sigma_0, k')}=0,~~\forall k\neq k'\in \Lambda.$
We immediately have that
 \begin{align}\label{P=0}
     \frac{1}{|\TTT^3|}\int_{\TTT^3}\phi^2_{(\epsilon_0, 1-\epsilon_0, k)}(x)\dd x=1, \quad \frac{1}{|\widetilde{\mathbb{T}}|}\int_{\widetilde{\mathbb{T}}}\widetilde g^2_{(2, \sigma_0, k)}(t)\dd t=1.
 \end{align}
 Thanks to $\phi_r=r^2\frac{\dd^2}{\dd x^2}\psi_r$ and $|k|=1$, we obtain that
  \begin{align}\nonumber
    \phi_{(\epsilon_0, 1-\epsilon_0, k)}(x)&=N^{-1}_{\Lambda}\lambda^{-1-\epsilon_0}_{q+1}\Div\Big(\big((\psi_{\lambda^{-\epsilon_0}_{q+1}})'(\lambda^{1-\epsilon_0}_{q+1}N_{\Lambda}k\cdot x)\big)k\Big).
 \end{align}
 In the following, we denote
 \begin{align*}
     \lambda^{-\epsilon_0}_{q+1}\big((\psi_{\lambda^{-\epsilon_0}_{q+1}})'(\lambda^{1-\epsilon_0}_{q+1}N_{\Lambda}k\cdot x)\big)=:\psi'_{\lambda^{-\epsilon_0}_{q+1}}(\lambda^{1-\epsilon_0}_{q+1}N_{\Lambda}k\cdot x)\big).
 \end{align*}
 Therefore, we have the relation
\begin{align}\nonumber
    \phi_{(\epsilon_0, 1-\epsilon_0, k)}(x)
   & =N^{-1}_{\Lambda}\lambda^{-1}_{q+1}\Div\big(\psi'_{\lambda^{-\epsilon_0}_{q+1}}(\lambda^{1-\epsilon_0}_{q+1}N_{\Lambda}k\cdot x)k\big)\nonumber
 \end{align}
Actually, $\phi_{(\epsilon_0, 1-\epsilon_0, k)}\bar{k}$ and $\phi_{(\epsilon_0, 1-\epsilon_0, k)}\bar{\bar{k}}$ are  the so-called \emph{intermittent shear flows}, which are introduced in \cite{2Beekie}. Moreover, it is well-known that the following estimates hold.
 \begin{proposition}[\cite{2Beekie}]\label{guji1}
For $p\in[1,\infty]$ and $m\in\mathbb{N}$, we have the following estimates
\begin{align*}
   &\big\|{D}^m\phi_{(\epsilon_0, 1-\epsilon_0, k)}\big\|_{L^p(\TTT^3)}\lesssim N^m_{\Lambda}{K^{\frac{1}{2}+\frac{2}{p}}}\lambda^{m-\epsilon_0(\frac{1}{p}-\frac{1}{2})}_{q+1}, \\
   &\big\|\tfrac{\dd ^m}{\dd t^m}g_{{(2, \sigma_0,k)}}\big\|_{L^p([0,T])}\lesssim \lambda^{{m(\sigma_0 +2)}-2(\frac{1}{p}-\frac{1}{2})}_{q+1}\\
 &\big\|{D}^m\psi'_{\lambda^{-\epsilon_0}_{q+1}}(\lambda^{1-\epsilon_0}_{q+1}N_{\Lambda}k\cdot x)\big\|_{L^p(\TTT^3)}\lesssim N^m_{\Lambda}{K^{\frac{1}{2}+\frac{2}{p}}}\lambda^{m-\epsilon_0(\frac{1}{p}-\frac{1}{2})}_{q+1}.
\end{align*}
\end{proposition}
\noindent\emph{Amplitudes}\quad Before constructing the perturbation, we define amplitudes firstly. Let $\chi:[0,\infty)\to [1,\infty)$ be a smooth function satisfying
\begin{equation}\label{def-chi}
\chi(z)=\left\{ \begin{alignedat}{-1}
&1, \quad 0\le z\le 1,\\
&z, \quad z\ge 2,
\end{alignedat}\right.
\end{equation}
with $z\le 2\chi(z)\le 4z$ for $z\in (1,2)$. We define that
\[\chi_q:=\frac{1}{TK^3}\chi\Big(\Big\langle\tfrac{\RR_{\ell_q}}{\delta_{q+1}\lambda^{-4\alpha}_q}\Big\rangle\Big)
.\]
Next, the space cutoffs $\{\widetilde\eta_q(x)\}_{q\ge 1}\in C^{\infty}_c(\R^3)$ are defined by
\begin{equation}\label{eta}
{\rm supp} \,\,\widetilde\eta_{q}=\Omega_{q+1} , \quad\quad   \widetilde\eta_q|_{\Omega_{q}+ [ -\lambda^{-1}_q ,\lambda^{-1}_q ]^3}\equiv 1,\quad~~\quad|D^N\widetilde\eta_q|\lesssim \lambda^{N}_q,
\end{equation}
where $\Omega_{q+1}$ is defined in \eqref{omega}. We introduce the time-space cutoffs  $\{\eta_q\}_{q\ge 1}$ by
\begin{align}\label{def-etaq}
    \eta_q(t,x)=(1-\zeta_q(t))\widetilde\eta_q(x),
\end{align}
where $\{ \zeta_q(t)\}\in C^\infty(\R)$ is defined by \eqref{def-zeta}.

Thanks to \eqref{supp-vlq} and \eqref{supp-Rlq}, one can easily deduce that
\begin{align}\label{R=Rq}
\eta_{q}\RR_{\ell_{q}}=\RR_{ \ell_{q}}
\end{align}
We claim that
\begin{align}\label{eta_qchi_q}
   \int_{\frac{3T}{4}}^T\int_{\R^3}\eta^2_q\chi_q\dd x\dd t
   \in \Big[\frac{1}{8}, 2\Big].
\end{align}
Indeed, we deduce from the definitions of $\eta_q$ and $\chi_q$ that
\begin{align*}
   \int_{\frac{3T}{4}}^T\int_{\R^3}\eta^2_q\chi_q\dd x\dd t=\frac{1}{TK^3}\int_{\frac{3T}{4}}^T\int_{\Omega_{q+1}}\widetilde\eta_q^2(x)\chi\Big(\Big\langle\tfrac{\RR_{\ell_q}}{\delta_{q+1}\lambda^{-4\alpha}_q}\Big\rangle\Big)\dd x\dd t.
\end{align*}
On one hand, by \eqref{e:R_ell}, one has
\begin{align}
   \frac{1}{TK^3}\int_{\frac{3T}{4}}^T\int_{\Omega_{q+1}}\widetilde\eta_q^2(x)\chi\Big(\Big\langle\tfrac{\RR_{\ell_q}}{\delta_{q+1}\lambda^{-4\alpha}_q}\Big\rangle\Big)\dd x\dd t
   \le&\frac{1}{TK^3}\int_{\frac{3T}{4}}^T\int_{\Omega_{q+1}}4\big(1+\big|\tfrac{\RR_{\ell_q}}{\delta_{q+1}\lambda^{-4\alpha}_q}\big|\big)\dd x\dd t\nonumber\\
   \le&\frac{1}{TK^3}(TK^3+4)\le 2. \nonumber
\end{align}
On the other hand,
\begin{align*}
   &\frac{1}{TK^3}\int_{\frac{3T}{4}}^T\int_{\Omega_{q+1}}\widetilde\eta_q^2(x)\chi\Big(\Big\langle\tfrac{\RR_{\ell_q}}{\delta_{q+1}\lambda^{-4\alpha}_q}\Big\rangle\Big)\dd x\dd t
   \ge\frac{1}{TK^3}\int_{\frac{3T}{4}}^T\int_{\Omega_{q}}1\dd x\dd t\ge\frac{1}{8}.
\end{align*}
We set
\begin{align}\label{rho-bar}
    \bar{\rho}_q=\frac{{E-\int_{\frac{3T}{4}}^T\int_{\R^3}|\vv_q|^2\dd x\dd t-3\delta_{q+2}}}{\int_{\frac{3T}{4}}^T\int_{\TTT^3}\eta^2_q\chi_q\dd x\dd t}.
\end{align}
Combining \eqref{e:E-vvq} with \eqref{eta_qchi_q} shows that
\begin{align*}
\bar{\rho}_q\in \Big[\frac{\delta_{q+1}}{8}, \frac{2\delta_{q+1}}{5}\Big].
\end{align*}
We give the temporal cutoff $\eta(t)\in C^{\infty}([0,T])$  by
\begin{equation}\label{def-eta}
    \eta(t)=1, \,t\in [0, \tfrac{T}{4}]; \quad\eta(t)=0, \,t\in [\tfrac{3T}{4}, T]; \quad|\partial^N_t\eta|\le CT^{-N},
\end{equation}
and  define
\begin{align}\label{rho-q}
    \rho_q(t)=\delta_{q+1}\eta(t)+(1-\eta(t))\bar{\rho}_q.
\end{align}

We are in position to give the  amplitudes $a_{(k,q)}(t,x)$ by
\begin{equation}\label{def-akq}
    a_{(k,q)}(t,x)=\eta_{q}
 a_k\Big({\rm Id}-\frac{\RR_{\ell_{q}}}{\chi_{q}\rho_q}\Big)
 (\chi_q \rho_q)^{1/2},
\end{equation}
 where $a_k$ stems from Lemma \ref{first S}.

\subsubsection{Construction the perturbation}We  construct the principal perturbation $ \wpq$ as follows:
\begin{equation}\label{def-wpq-1}
 \begin{aligned}
 \wpq&:=\sum_{k\in\Lambda}a_{(k,q)} \phi_{(\epsilon_0, 1-\epsilon_0, k)}(x) g_{(2, \sigma_0, k)}(t)\bar{k}.
 \end{aligned}
 \end{equation}
Since $\wpq$ is not divergence-free,  we need to construct the incompressibility corrector of  $\wpq$  to ensure that the perturbation is divergence-free. Noting  that  for any $k\perp \bar{k}$, we have
\begin{align*}
  {\Div\big(\psi'_{\lambda^{-\epsilon_0}_{q+1}}(\lambda^{1-\epsilon_0}_{q+1}N_{\Lambda}k\cdot x)\bar{k}\otimes k\big)}=0,
\end{align*}
and
\begin{align*}
    \phi_{(\epsilon_0, 1-\epsilon_0, k)}(x) \bar{k}=N^{-1}_{\Lambda}\lambda^{-1}_{q+1}{\Div\big(\psi'_{\lambda^{-\epsilon_0}_{q+1}}(\lambda^{1-\epsilon_0}_{q+1}N_{\Lambda}k\cdot x)k\otimes \bar{k}\big)}.
\end{align*}
Then we rewrite $\wpq$ as
\begin{align}
 \wpq=&\sum_{k\in\Lambda}N^{-1}_{\Lambda}\lambda^{-1}_{q+1}a_{(k,q)}   g_{(2, \sigma_0, k)}(t){\Div\big(\psi'_{\lambda^{-\epsilon_0}_{q+1}}(\lambda^{1-\epsilon_0}_{q+1}N_{\Lambda}k\cdot x)k\otimes \bar{k}\big)}\nonumber\\
 =&\sum_{k\in\Lambda}N^{-1}_{\Lambda}\lambda^{-1}_{q+1}a_{(k,q)}  g_{(2, \sigma_0, k)}(t){\Div\big(\psi'_{\lambda^{-\epsilon_0}_{q+1}}(\lambda^{1-\epsilon_0}_{q+1}N_{\Lambda}k\cdot x)(k\otimes \bar{k}-\bar{k}\otimes k)\big)}. \label{def-wpq}
\end{align}
Based on this equality, we define $\wcq$ by
\begin{align}\label{def-wcq}
\wcq=&\sum_{k\in\Lambda}N^{-1}_{\Lambda}\lambda^{-1}_{q+1} \psi'_{\lambda^{-\epsilon_0}_{q+1}}(\lambda^{1-\epsilon_0}_{q+1}N_{\Lambda}k\cdot x) g_{(2, \sigma_0, k)}(t){\Div\big(a_{(k,q)}  (k\otimes \bar{k}-\bar{k}\otimes k)\big)}.
\end{align}
One easily deduces that
\begin{align*}
&\Div(\wpq+\wcq)\\
=&\sum_{k\in\Lambda}N^{-1}_{\Lambda}\lambda^{-1}_{q+1}\Div\Div\big(a_{(k,q)}   g_{(2, \sigma_0, k)}(t)\psi'_{\lambda^{-\epsilon_0}_{q+1}}(\lambda^{1-\epsilon_0}_{q+1}N_{\Lambda}k\cdot x)(k\otimes \bar{k}-\bar{k}\otimes k)\big)
=0,
\end{align*}
where we have used the fact that $\Div\Div M=0$ for any anti-symmetric matrix $M$.

Finally, we define $\wtq$ by solving the following Cauchy problem
 \begin{equation}
\left\{ \begin{alignedat}{-1}
&\del_t \wtq-\Delta \wtq+\Div (\wtq\otimes \wtq)+\Div(\wtq\otimes\ubqnl)\\
&\qquad\qquad\qquad\qquad\quad+\Div(\ubqnl\otimes \wtq)  +\nabla p_{t}   = F_{q+1},
\\
 & \Div \wtq = 0,
  \\
  & \wtq |_{t=0}=  0 ,
\end{alignedat}\right.
 \label{e:wt}
\end{equation}
where the force term $F_{q+1}$ is determined by $\RR^{rem}_q$, $(u_q-u_{\ell_q})$ and some terms extracted from $\partial_t (\wpq+\wcq)$ and $\Div(\wpq\otimes \wpq)$,  as  in \eqref{def-F_{q+1}}. To demonstrate the derivation of $F_{q+1}$ clearly, we decompose $\partial_t (\wpq+\wcq)$ and $\Div(\wpq\otimes \wpq)$ firstly.
\begin{proposition}[The decomposition of $\partial_t (\wpq+\wcq)$]\label{F1}Let
$\wpq$ and $\wcq$ be defined in \eqref{def-wpq} and \eqref{def-wcq} respectively, we have
\begin{align*}
    \partial_t (\wpq+\wcq)=\Div \RR^{(1)}_{q+1}+\Div \RR^{(2)}_{q+1}+F^{(1)}_{q+1}+\nabla P^{(1)}_{q+1},
\end{align*}
where
\begin{align*}
&\RR^{(1)}_{q+1}=N^{-1}_{\Lambda}\lambda^{-1}_{q+1}\sum_{k\in\Lambda}{\partial_t(a_{(k,q)}  g_{(2, \sigma_0, k)}(t))\psi'_{\lambda^{-\epsilon_0}_{q+1}}(\lambda^{1-\epsilon_0}_{q+1}N_{\Lambda}k\cdot x)(k\otimes \bar{k}+\bar{k}\otimes k)},\\
&\RR^{(2)}_{q+1}=-2N^{-2}_{\Lambda}\lambda^{-2}_{q+1}\sum_{k\in\Lambda}\psi_{\lambda^{-\epsilon_0}_{q+1}} (\lambda^{1-\epsilon_0}_{q+1}N_{\Lambda}k\cdot x)k\ootimes{\Div\big(\partial_t(a_{(k,q)}   g_{(2, \sigma_0, k)}(t))\bar{k}\otimes k\big)},\\
&F^{(1)}_{q+1}=2N^{-2}_{\Lambda}\lambda^{-2}_{q+1}\sum_{k\in\Lambda}\psi_{\lambda^{-\epsilon_0}_{q+1}} (\lambda^{1-\epsilon_0}_{q+1}N_{\Lambda}k\cdot x)k\cdot\nabla {\Div\big(\partial_t(a_{(k,q)}  g_{(2, \sigma_0, k)}(t))\bar{k}\otimes k\big)},\\
&P^{(1)}_{q+1}=-\frac{2}{3}N^{-2}_{\Lambda}\lambda^{-2}_{q+1}\sum_{k\in\Lambda}\psi_{\lambda^{-\epsilon_0}_{q+1}} (\lambda^{1-\epsilon_0}_{q+1}N_{\Lambda}k\cdot x) \Div\big(\partial_t(a_{(k,q)}   g_{(2, \sigma_0, k)}(t) \bar{k}\big).
\end{align*}
\end{proposition}
\begin{proof}Thanks to \eqref{def-wpq} and \eqref{def-wcq}, we have
\begin{align}\label{wpq+wcq}
\wpq+\wcq=\sum_{k\in\Lambda}N^{-1}_{\Lambda}\lambda^{-1}_{q+1}{\Div\big(a_{(k,q)}  g_{(2, \sigma_0, k)}(t)\psi'_{\lambda^{-\epsilon_0}_{q+1}}(\lambda^{1-\epsilon_0}_{q+1}N_{\Lambda}k\cdot x)(k\otimes \bar{k}-\bar{k}\otimes k)\big)}.
    \end{align}
Therefore, one infers that
\begin{align}
\partial_t(\wpq+\wcq)=&\sum_{k\in\Lambda}N^{-1}_{\Lambda}\lambda^{-1}_{q+1}{\Div\big(\partial_t(a_{(k,q)}   g_{(2, \sigma_0, k)}(t))\psi'_{\lambda^{-\epsilon_0}_{q+1}}(\lambda^{1-\epsilon_0}_{q+1}N_{\Lambda}k\cdot x)(k\otimes \bar{k}-\bar{k}\otimes k)\big)}\nonumber\\
=&\sum_{k\in\Lambda}N^{-1}_{\Lambda}\lambda^{-1}_{q+1}{\Div\big(\partial_t(a_{(k,q)}g_{(2, \sigma_0, k)}(t))\psi'_{\lambda^{-\epsilon_0}_{q+1}}(\lambda^{1-\epsilon_0}_{q+1}N_{\Lambda}k\cdot x)(k\otimes \bar{k}+\bar{k}\otimes k)\big)}\nonumber\\
&-2\sum_{k\in\Lambda}N^{-1}_{\Lambda}\lambda^{-1}_{q+1}{\Div\big(\partial_t(a_{(k,q)}   g_{(2, \sigma_0, k)}(t))\psi'_{\lambda^{-\epsilon_0}_{q+1}}(\lambda^{1-\epsilon_0}_{q+1}N_{\Lambda}k\cdot x)\bar{k}\otimes k\big)}.\label{ptwpq+wcq}
\end{align}
For the second term, due to $k\perp \bar{k}$ and $|k|=1$, one deduces that
\begin{align*}
&-2\sum_{k\in\Lambda}N^{-1}_{\Lambda}\lambda^{-1}_{q+1}{\Div\big(\partial_t(a_{(k,q)}   g_{(2, \sigma_0, k)}(t))\psi'_{\lambda^{-\epsilon_0}_{q+1}}(\lambda^{1-\epsilon_0}_{q+1}N_{\Lambda}k\cdot x)\bar{k}\otimes k\big)}\\
=&-2\sum_{k\in\Lambda}N^{-1}_{\Lambda}\lambda^{-1}_{q+1}\psi'_{\lambda^{-\epsilon_0}_{q+1}}(\lambda^{1-\epsilon_0}_{q+1}N_{\Lambda}k\cdot x){\Div\big(\partial_t(a_{(k,q)}  g_{(2, \sigma_0, k)}(t))\bar{k}\otimes k\big)}\\
=&-2\sum_{k\in\Lambda}N^{-2}_{\Lambda}\lambda^{-2}_{q+1}\Div\big(\psi_{\lambda^{-\epsilon_0}_{q+1}} (\lambda^{1-\epsilon_0}_{q+1}N_{\Lambda}k\cdot x)k\big){\Div\big(\partial_t(a_{(k,q)}  g_{(2, \sigma_0, k)}(t))\bar{k}\otimes k\big)}\\
=&-2N^{-2}_{\Lambda}\lambda^{-2}_{q+1}\sum_{k\in\Lambda}\Div\big(\psi_{\lambda^{-\epsilon_0}_{q+1}} (\lambda^{1-\epsilon_0}_{q+1}N_{\Lambda}k\cdot x)k\otimes{\Div\big(\partial_t(a_{(k,q)}   g_{(2, \sigma_0, k)}(t))\bar{k}\otimes k\big)}\big)\\
&+2N^{-2}_{\Lambda}\lambda^{-2}_{q+1}\sum_{k\in\Lambda}\psi_{\lambda^{-\epsilon_0}_{q+1}} (\lambda^{1-\epsilon_0}_{q+1}N_{\Lambda}k\cdot x)k\cdot\nabla {\Div\big(\partial_t(a_{(k,q)}  g_{(2, \sigma_0, k)}(t))\bar{k}\otimes k\big)}\\
=&-2N^{-2}_{\Lambda}\lambda^{-2}_{q+1}\sum_{k\in\Lambda}\Div\big(\psi_{\lambda^{-\epsilon_0}_{q+1}} (\lambda^{1-\epsilon_0}_{q+1}N_{\Lambda}k\cdot x)k\ootimes{\Div\big(\partial_t(a_{(k,q)}   g_{(2, \sigma_0, k)}(t))\bar{k}\otimes k\big)}\big)\\
&-\frac{2}{3}N^{-2}_{\Lambda}\lambda^{-2}_{q+1}\nabla\Big(\sum_{k\in\Lambda}\psi_{\lambda^{-\epsilon_0}_{q+1}} (\lambda^{1-\epsilon_0}_{q+1}N_{\Lambda}k\cdot x) \Div\big(\partial_t(a_{(k,q)}   g_{(2, \sigma_0, k)}(t))\bar{k}\big)\Big)\\
&+2N^{-2}_{\Lambda}\lambda^{-2}_{q+1}\sum_{k\in\Lambda}\psi_{\lambda^{-\epsilon_0}_{q+1}} (\lambda^{1-\epsilon_0}_{q+1}N_{\Lambda}k\cdot x)k\cdot\nabla {\Div\big(\partial_t(a_{(k,q)}  g_{(2, \sigma_0, k)}(t))\bar{k}\otimes k\big)}
\end{align*}
where we have used the fact that $v\ootimes u:=v\otimes u-\frac{1}{3}v\cdot u {\rm Id}$. Plugging this equality into \eqref{ptwpq+wcq} yields Proposition \ref{F1}.
\end{proof}

\begin{proposition}[The decomposition of $\Div(\wpq\otimes\wpq)$]\label{def-F2}Let
$\wpq$ be defined in \eqref{def-wpq-1}, we have
\begin{align*}
\Div(\wpq\otimes \wpq)+\Div\RR_{\ell_q}=\Div \RR^{(3)}_{q+1}+F^{(2)}_{q+1}+F^{(3)}_{q+1}+\nabla\big(\eta^2_q (\chi_q \rho_q)^{1/2}\big),
\end{align*}
where
\begin{align*}
&\RR^{(3)}_{q+1}=-\ii\sum_{k\in\Lambda}\sum_{m,l \in\ZZ\backslash\{0\}, m+l\neq 0}
\frac{Kc_{l, \epsilon_0}c_{m, \epsilon_0}e^{\ii2\pi  \lambda^{1-\epsilon_0}_{q+1} (l+m)  N_{\Lambda}k\cdot x/K }}{2\pi\lambda^{1-\epsilon_0}_{q+1}(l+m)} k\otimes \Div\big(a^2_{(k,q)}  g^2_{(2, \sigma_0, k)}\bar{k}\otimes \bar{k}\big)\\
&\quad\quad\quad-\ii\sum_{k\in\Lambda}\sum_{m,l \in\ZZ\backslash\{0\}, m+l\neq 0}
\frac{Kc_{l, \epsilon_0}c_{m, \epsilon_0}e^{\ii2\pi  \lambda^{1-\epsilon_0}_{q+1} (l+m)  N_{\Lambda}k\cdot x/K }}{2\pi\lambda^{1-\epsilon_0}_{q+1}(l+m)} \bar{k}\otimes \Div\big(a^2_{(k,q)}  g^2_{(2, \sigma_0, k)}\bar{k}\otimes {k}\big),\\
&F^{(2)}_{q+1}=\ii\sum_{k\in\Lambda}\sum_{m,l\in\ZZ\backslash\{0\}, m+l\neq 0}
\frac{Kc_{l, \epsilon_0}c_{m, \epsilon_0}e^{\ii 2\pi  \lambda^{1-\epsilon_0}_{q+1} (l+m)N_{\Lambda} k\cdot x /K}}{2\pi N_{\Lambda}\lambda^{1-\epsilon_0}_{q+1}(l+m)} k\cdot\nabla \Div\big(a^2_{(k,q)}g^2_{(2, \sigma_0, k)} \bar{k}\otimes \bar{k}\big)\\
&\quad\quad\quad+\ii\sum_{k\in\Lambda}\sum_{m,l\in\ZZ\backslash\{0\}, m+l\neq 0}
\frac{Kc_{l, \epsilon_0}c_{m, \epsilon_0}e^{\ii 2\pi  \lambda^{1-\epsilon_0}_{q+1} (l+m)N_{\Lambda} k\cdot x /K}}{2\pi N_{\Lambda}\lambda^{1-\epsilon_0}_{q+1}(l+m)}\bar{k}\cdot\nabla \Div\big(a^2_{(k,q)}g^2_{(2, \sigma_0, k)} \bar{k}\otimes  k\big),\\
&F^{(3)}_{q+1}=\sum_{k\in\Lambda} \Div\big(a^2_{(k,q)} \mathbb{P}_{\neq 0}(g^2_{(2, \sigma_0, k)}) \bar{k}\otimes \bar{k}\big).
\end{align*}
\end{proposition}
\begin{proof} Thanks to \eqref{P=0}, one has
\begin{align*}
\mathbb{P}_{=0}(\phi^2_{(\epsilon_0, 1-\epsilon_0, k)})=1,\quad\text{and}\quad
    \mathbb{P}_{=0}(g^2_{(2, \sigma_0, k)}(t))=1.
\end{align*}
According to \eqref{def-wpq-1},
we obtain that
\begin{align*}
\Div(\wpq\otimes \wpq)+\Div\RR_{\ell_q}
=&\sum_{k\in\Lambda}\Div\big(a^2_{(k,q)} \phi^2_{(\epsilon_0, 1-\epsilon_0, k)} g^2_{(2, \sigma_0, k)}\bar{k}\otimes \bar{k}\big)\\
=&\sum_{k\in\Lambda} \Div\big(a^2_{(k,q)} \bar{k}\otimes \bar{k}\big)+\Div\RR_{\ell_q}+\sum_{k\in\Lambda} \Div\big(a^2_{(k,q)} \mathbb{P}_{\neq 0}(g^2_{(2, \sigma_0, k)}) \bar{k}\otimes \bar{k}\big)\\
&+\sum_{k\in\Lambda}\Div\big(a^2_{(k,q)} g^2_{(2, \sigma_0, k)}\mathbb{P}_{\neq0}(\phi^2_{(\epsilon_0, 1-\epsilon_0, k)}) \bar{k}\otimes \bar{k}\big).
\end{align*}
By  the definition of $a_{(k,q)}$ in \eqref{def-akq} and Lemma \ref{first S}, we have
\begin{align*}
&\sum_{k\in\Lambda} \Div\big(a^2_{(k,q)} \bar{k}\otimes \bar{k}\big)+\Div\RR_{\ell_q}=\Div\big( \eta^2_q(\chi_q \rho_q)^{1/2}{\rm Id }-\eta^2_q\RR_{\ell_q}\big)+\Div\RR_{\ell_q}
=\nabla\big(\eta^2_q (\chi_q \rho_q)^{1/2}\big),
\end{align*}
where we have used \eqref{R=Rq}. By Proposition \ref{est-cmr}, one has
\begin{align*}
&\sum_{k\in\Lambda} \Div\big(a^2_{(k,q)} g^2_{(2, \sigma_0, k)}\mathbb{P}_{\neq0}(\phi^2_{(\epsilon_0, 1-\epsilon_0, k)}) \bar{k}\otimes \bar{k}\big)\\
=&\sum_{k\in\Lambda} \Div\big(a^2_{(k,q)}g^2_{(2, \sigma_0, k)}\bar{k}\otimes \bar{k}\sum_{m,l\in \ZZ  \backslash\{0\}, m+l\neq 0} c_{l, \epsilon_0}c_{m, \epsilon_0}e^{\ii 2\pi \lambda^{1-\epsilon_0}_{q+1} (l+m)N_{\Lambda} k\cdot x/K } \big)\\
=&\sum_{k\in\Lambda} \Div\big(a^2_{(k,q)} g^2_{(2, \sigma_0, k)}\bar{k}\otimes \bar{k}\big)\sum_{m,l\in \ZZ  \backslash\{0\}, m+l\neq 0} c_{l, \epsilon_0}c_{m, \epsilon_0}e^{\ii 2\pi \lambda^{1-\epsilon_0}_{q+1} (l+m) N_{\Lambda}  k\cdot x /K } \\
=&-\ii\sum_{k\in\Lambda} \Div\big(a^2_{(k,q)} g^2_{(2, \sigma_0, k)}\bar{k}\otimes \bar{k}\big)\sum_{m,l\in \ZZ  \backslash\{0\}, m+l\neq 0}\frac{K\Div\big(c_{l, \epsilon_0}c_{m, \epsilon_0}e^{\ii2\pi\lambda^{1-\epsilon_0}_{q+1} (l+m) N_{\Lambda} k\cdot x /K}k\big)}{2\pi N_{\Lambda}\lambda^{1-\epsilon_0}_{q+1}(l+m)}\\
=&\Div\Big(-\ii\sum_{k\in\Lambda}\sum_{m,l \in\ZZ\backslash\{0\}, m+l\neq 0}
\frac{Kc_{l, \epsilon_0}c_{m, \epsilon_0}e^{\ii2\pi  \lambda^{1-\epsilon_0}_{q+1} (l+m)  N_{\Lambda}k\cdot x/K }}{2\pi\lambda^{1-\epsilon_0}_{q+1}(l+m)} k\otimes \Div\big(a^2_{(k,q)}  g^2_{(2, \sigma_0, k)}\bar{k}\otimes \bar{k}\big)\\
&{-\ii\sum_{k\in\Lambda}\sum_{m,l \in\ZZ\backslash\{0\}, m+l\neq 0}
\frac{Kc_{l, \epsilon_0}c_{m, \epsilon_0}e^{\ii2\pi  \lambda^{1-\epsilon_0}_{q+1} (l+m)  N_{\Lambda}k\cdot x/K }}{2\pi\lambda^{1-\epsilon_0}_{q+1}(l+m)} \bar{k}\otimes \Div\big(a^2_{(k,q)}  g^2_{(2, \sigma_0, k)}\bar{k}\otimes {k}\big)}\Big)\\
&+\ii\sum_{k\in\Lambda}\sum_{m,l\in\ZZ\backslash\{0\}, m+l\neq 0}
\frac{Kc_{l, \epsilon_0}c_{m, \epsilon_0}e^{\ii 2\pi  \lambda^{1-\epsilon_0}_{q+1} (l+m)N_{\Lambda} k\cdot x /K}}{2\pi N_{\Lambda}\lambda^{1-\epsilon_0}_{q+1}(l+m)} k\cdot\nabla \Div\big(a^2_{(k,q)}g^2_{(2, \sigma_0, k)} \bar{k}\otimes \bar{k}\big)\\
&+{\ii\sum_{k\in\Lambda}\sum_{m,l\in\ZZ\backslash\{0\}, m+l\neq 0}
\frac{Kc_{l, \epsilon_0}c_{m, \epsilon_0}e^{\ii 2\pi  \lambda^{1-\epsilon_0}_{q+1} (l+m)N_{\Lambda} k\cdot x /K}}{2\pi N_{\Lambda}\lambda^{1-\epsilon_0}_{q+1}(l+m)} \bar{k}\cdot\nabla \Div\big(a^2_{(k,q)}g^2_{(2, \sigma_0, k)} \bar{k}\otimes {k}\big)}\\
=:&\Div R^{(3)}_{q+1}+F^{(2)}_{q+1}.
\end{align*}
Hence, we have
\begin{align*}
&\quad\,\,\Div(\wpq\otimes \wpq)+\Div\RR_{\ell_q}\\
&=\Div R^{(3)}_{q+1}+F^{(2)}_{q+1}+\sum_{k\in\Lambda} \Div\big(a^2_{(k,q)} \mathbb{P}_{\neq 0}(g^2_{(2, \sigma_0, k)}) \bar{k}\otimes \bar{k}\big)+\nabla\big(\eta^2_q (\chi_q \rho_q)^{1/2}\big)\\
&=:\Div R^{(3)}_{q+1}+F^{(2)}_{q+1}+F^{(3)}_{q+1}+\nabla\big(\eta^2_q (\chi_q \rho_q)^{1/2}\big).
\end{align*}
This completes the proof of Proposition \ref{def-F2}.
\end{proof}
Now we give the definition of $\wtq$ in more details. Let $\wtq$ be the solution of the equations \eqref{e:wt}, where  $F_{q+1}$ is given by
\begin{align}
    F_{q+1}:=&-F^{(1)}_{q+1}-F^{(2)}_{q+1}-F^{(3)}_{q+1}-(1-\zeta_q)\Div \RR^{rem}_q-\partial_t\zeta_q(u_q-u_{\ell_q})\nonumber\\
    &-\zeta_q(1-\zeta_q)\Div ((u_q-u_{\ell_q})\otimes(u_q-u_{\ell_q})). \label{def-F_{q+1}}
\end{align}
Then we define the perturbation $w_{q+1}$ by
\[w_{q+1}=\wpq+\wcq+\wtq.\]
\subsubsection{Estimates for the perturbation}In order to estimate  $w_{q+1}$, we give some  preliminary estimates in terms of $\chi_q\rho_q$, $a_{(k,q)}$ and $F_{q+1}$.

Note that  the smooth function $\chi \ge1$ in \eqref{def-chi}, one can expect to obtain higher-order derivative estimates for $(\chi_q\rho_q)^{1/2}$ and  $(\chi_q\rho_q)^{-1}$, which is shown by Proposition \ref{cr1/2} as follows.
{\begin{proposition}[Estimates for $(\chi_q\rho_q)^{1/2}$ and $(\chi_q\rho_q)^{-1}$]\label{cr1/2}For any $N\ge 2$, we have
\begin{align}
&\|(\chi_q\rho_q)^{1/2}\|_{L^\infty_tH^N}\le    C_N(TK^3)^{-1/2}\delta^{-N+1/2}_{q+1}\lambda^{5N+4N\alpha}_q\ell^{-N}_q, \label{chi-H4}\\
& \|(\chi_q\rho_q)^{-1}\|_{L^\infty_tH^N}
\le C_NTK^3\delta^{-N-1}_{q+1}\lambda^{5N+4N\alpha}_q\ell^{-N}_q,   \label{est-cr-1}\\
&\|\partial_t(\chi_q\rho_q)^{1/2}\|_{L^\infty_tH^N}\le  C_N(TK^3)^{-1/2}\delta^{-N-1/2}_{q+1}\lambda^{5N+5+4(N+1)\alpha}_q\ell^{-2N}_q,\label{t-chi-H4}\\
&\|\partial_t(\chi_q\rho_q)^{-1}\|_{L^\infty_tH^N}
\le C_NTK^3\delta^{-N-2}_{q+1}\lambda^{5N+5+4(N+1)\alpha}_q\ell^{-2N}_q. \label{est-t-cr-1}
\end{align}
\end{proposition}
\begin{proof}
Since $\chi(x)\ge1$, $\rho_q\sim \delta_{q+1}$  and
    \begin{align*}
    \chi_q\rho_q=\frac{1}{TK^3}\chi\Big(\Big\langle\tfrac{\RR_{\ell_q}}{\delta_{q+1}\lambda^{-4\alpha}_q}\Big\rangle\Big)\rho_q,
    \end{align*}
we infer from \eqref{e:R_ell-W} that
\begin{align*}
\|(\chi_q\rho_q)^{1/2}\|_{L^\infty_tH^N}=&\|\chi^{1/2}_q\|_{L^\infty_t H^N}\|\rho_q^{1/2}(t)\|_{L^\infty}\\
\le&C_N(TK^3)^{-1/2}\delta^{1/2}_{q+1}\Big(1+\Big{\|}\frac{\RR_{\ell_q}}{\delta_{q+1}\lambda^{-4\alpha}_q}\Big{\|}_{L^\infty_{t,x}}\Big)^{N-1}\Big{\|}\frac{\RR_{\ell_q}}{\delta_{q+1}\lambda^{-4\alpha}_q}\Big{\|}_{L^\infty_tH^N}\\
\le& C_N(TK^3)^{-1/2}\delta^{-N+1/2}_{q+1}\lambda^{5N+4N\alpha}_q\ell^{-N}_q
\end{align*}
and
\begin{align}
 \|(\chi_q\rho_q)^{-1}\|_{L^\infty_tH^N}\le& C_NTK^3\delta^{-1}_{q+1}\Big(1+\Big{\|}\frac{\RR_{\ell_q}}{\delta_{q+1}\lambda^{-4\alpha}_q}\Big{\|}_{L^\infty_{t,x}}\Big)^{N-1}\Big{\|}\frac{\RR_{\ell_q}}{\delta_{q+1}\lambda^{-4\alpha}_q}\Big{\|}_{L^\infty_tH^N}\nonumber\\
\le& C_NTK^3\delta^{-N-1}_{q+1}\lambda^{5N+4N\alpha}_q\ell^{-N}_q. \label{es:cr-1}
\end{align}
With the aid of \eqref{e:R_ell-W} and $T^{-1}\le \lambda_q$, one deduces  that
\begin{align*}
&\|\partial_t(\chi_q\rho_q)^{1/2}\|_{L^\infty_tH^N}\\
\le&\|\partial_t\chi_q^{1/2}\|_{L^\infty_tH^N}\|\rho_q^{1/2}(t)\|_{L^\infty}+
\|\chi_q^{1/2}\|_{L^\infty_tH^N}\|\partial_t\rho_q^{1/2}(t)\|_{L^\infty}\\
\le& C_N(TK^3)^{-1/2}\delta^{1/2}_{q+1}\Big{\|}\frac{\partial_t\RR_{\ell_q}}{\delta_{q+1}\lambda^{-4\alpha}_q}\Big{\|}_{L^\infty_tH^N}\Big(1+\Big{\|}\frac{\RR_{\ell_q}}{\delta_{q+1}\lambda^{-4\alpha}_q}\Big{\|}_{L^\infty_{t,x}}\Big)^{N-1}\Big{\|}\frac{\RR_{\ell_q}}{\delta_{q+1}\lambda^{-4\alpha}_q}\Big{\|}_{L^\infty_tH^N}\\
&+ C_N(TK^3)^{-1/2}T^{-1}\delta^{1/2}_{q+1}\Big(1+\Big{\|}\frac{\RR_{\ell_q}}{\delta_{q+1}\lambda^{-4\alpha}_q}\Big{\|}_{L^\infty_{t,x}}\Big)^{N-1}\Big{\|}\frac{\RR_{\ell_q}}{\delta_{q+1}\lambda^{-4\alpha}_q}\Big{\|}_{L^\infty_tH^N}\\
\le& C_N(TK^3)^{-1/2}\delta^{-N-1/2}_{q+1}\lambda^{5N+5+4(N+1)\alpha}_q\ell^{-2N}_q,
\end{align*}
and
\begin{align*}
&\|\partial_t(\chi_q\rho_q)^{-1}\|_{L^\infty_tH^N}\\
\le&\|\partial_t\chi_q^{-1}\|_{L^\infty_tH^N}\|\rho_q^{-1}(t)\|_{L^\infty}+
\|\chi_q^{-1}\|_{L^\infty_tH^N}\|\partial_t\rho_q^{-1}(t)\|_{L^\infty}\\
\le& C_NTK^3\delta^{-1}_{q+1}\Big{\|}\frac{\partial_t\RR_{\ell_q}}{\delta_{q+1}\lambda^{-4\alpha}_q}\Big{\|}_{L^\infty_tH^N}\Big(1+\Big{\|}\frac{\RR_{\ell_q}}{\delta_{q+1}\lambda^{-4\alpha}_q}\Big{\|}_{L^\infty_{t,x}}\Big)^{N-1}\Big{\|}\frac{\RR_{\ell_q}}{\delta_{q+1}\lambda^{-4\alpha}_q}\Big{\|}_{L^\infty_tH^N}\\
&+ C_NTK^3T^{-1}\delta^{-1}_{q+1}\Big(1+\Big{\|}\frac{\RR_{\ell_q}}{\delta_{q+1}\lambda^{-4\alpha}_q}\Big{\|}_{L^\infty_{t,x}}\Big)^{N-1}\Big{\|}\frac{\RR_{\ell_q}}{\delta_{q+1}\lambda^{-4\alpha}_q}\Big{\|}_{L^\infty_tH^N}\\
\le& C_NTK^3\delta^{-N-2}_{q+1}\lambda^{5N+5+4(N+1)\alpha}_q\ell^{-2N}_q.
\end{align*}
Hence we complete the proof of Proposition \ref{cr1/2}.
\end{proof}}

\begin{proposition}[Estimates for ${a}_{(k,q)}$]\label{est-ak}{For $N\ge 2$, we have
\begin{align}
&\|{a}_{(k,q)}\|_{L^\infty_t H^N}\le C_N\ell^{-2N}_q, \label{akq-HN} \\
&\|\partial_t {a}_{(k,q)}\|_{L^\infty_t H^N}\le  C_N\ell^{-6N}_q.\label{t-akq-HN}
%&\|\partial^2_t {a}_{(k,q)}\|_{L^\infty_T H^N}\le  C_N\delta^{-7N}_{q+1}\lambda^{30N}_{q}\ell^{-10N}_q.\label{tt-akq-HN}
\end{align}}
\end{proposition}
\begin{proof}{
Combining with \eqref{e:R_ell-W} and \eqref{est-cr-1}, we obtain that
\begin{align}
\Big\|\frac{\RR_{\ell_q}}{\chi_q \rho_q}\Big\|_{L^\infty_tH^N}\le& \|\RR_{\ell_q}\|_{L^\infty_{t,x}} \|(\chi_q\rho_q)^{-1}\|_{L^\infty_t H^N}+\|\RR_{\ell_q}\|_{L^\infty_t H^N} \|(\chi_q\rho_q)^{-1}\|_{L^\infty_{t,x}}\nonumber\\
\le&C_N TK^3\delta^{-N-1}_{q+1}\lambda^{5N+4N\alpha}_q\ell^{-N}_q+CTK^3\delta^{-1}_{q+1}\lambda^{5}_q\ell^{-N}_q\nonumber\\
\le&C_N TK^3\delta^{-N-1}_{q+1}\lambda^{5N+4N\alpha}_q\ell^{-N}_q.\label{R-cr-HN}
\end{align}
This inequality together with \eqref{e:R_ell-W} implies that
\begin{align}
 \Big{\|} a_{k}\Big({\rm Id}-\frac{\RR_{\ell_q}}{\chi_q\rho_q}\Big)\Big{\|}_{L^\infty_t H^N}\le &C_N\Big(1+\Big\|\frac{\RR_{\ell_q}}{\chi_q\rho_q}\Big\|_{L^\infty_{t,x}}\Big)^{N-1}\Big\|\frac{\RR_{\ell_q}}{\chi_q\rho_q}\Big\|_{L^\infty_tH^N}\nonumber\\
\le&C_N (TK^3)^{N}\delta^{-2N}_{q+1}\lambda^{10N+4N\alpha}_q\ell^{-N}_q.\label{ak-HN}
\end{align}
Since ${a}_{(k,q)}=\eta_q a_{k}\big({\rm Id}-\frac{\RR_{\ell_q}}{\chi_q\rho_q}\big)(\chi_q\rho_q)^{1/2}$, we have by \eqref{chi-H4}, \eqref{R-cr-HN} and \eqref{ak-HN} that
\begin{align*}
\|{a}_{(k,q)}\|_{L^\infty_t H^N}\le& \Big{\|} a_{k}\Big({\rm Id}-\frac{\RR_{\ell_q}}{\chi_q\rho_q}\Big)\Big{\|}_{L^\infty_t H^N}\|(\chi_q\rho_q)^{1/2}\|_{L^\infty_{t,x}}+C\|(\chi_q\rho_q)^{1/2}\|_{L^\infty_t H^N}\\
\le&C_N (TK^3)^{N}\delta^{-2N}_{q+1}\lambda^{10N+4N\alpha}_q\ell^{-N}_q+C_N(TK^3)^{-1/2}\delta^{-N+1/2}_{q+1}\lambda^{5N+4N\alpha}_q\ell^{-N}_q\\
\le&C_N (TK^3)^{N}\delta^{-2N}_{q+1}\lambda^{10N+4N\alpha}_q\ell^{-N}_q.
\end{align*}
Owning to
\begin{align}\label{Con1}
T<1, \,\,\alpha<\frac{1}{4}, \,\,K^3<\lambda_q, \,\,\delta^{-4}_{q+1}<\lambda^{\alpha}_q, \,\,\ell_q=\lambda^{-50}_q,
\end{align}
 we obtain \eqref{akq-HN} by the above estimate.}

By \eqref{e:R_ell-W}, \eqref{est-cr-1} and \eqref{est-t-cr-1}, we have that
\begin{align}
\Big\|\partial_t\Big(\frac{\RR_{\ell_q}}{\chi_q \rho_q}\Big)\Big\|_{L^\infty_tH^N}\le& \|\partial_t\RR_{\ell_q}\|_{L^\infty_t H^N} \|(\chi_q\rho_q)^{-1}\|_{L^\infty_t H^N}+\|\RR_{\ell_q}\|_{L^\infty_t H^N} \|\partial_t(\chi_q\rho_q)^{-1}\|_{L^\infty_t H^N}\nonumber\\
\le&C_N TK^3\delta^{-N-2}_{q+1}\lambda^{5N+10+4(N+1)\alpha}_q\ell^{-3N}_q\nonumber\\
&+C_NTK^3\delta^{-N-2}_{q+1}\lambda^{5N+10+4(N+1)\alpha}_q\ell^{-3N+1}_q\nonumber\\
\le&C_N TK^3\delta^{-N-2}_{q+1}\lambda^{5N+10+4(N+1)\alpha}_q\ell^{-3N}_q.\label{t-R-cr-HN}
\end{align}
Therefore, one infers from  \eqref{e:R_ell-W}, \eqref{R-cr-HN} and \eqref{t-R-cr-HN} that
\begin{align}
\Big\|\partial_ta_{k}\Big({\rm Id}-\frac{\RR_{\ell_q}}{\chi_q \rho_q}\Big)\Big\|_{L^\infty_tH^N}\le&\Big\|(a'_k)\Big({\rm Id}-\frac{\RR_{\ell_q}}{\chi_q \rho_q}\Big)\Big\|_{L^\infty_tH^N}\Big\|\partial_t\Big(\frac{\RR_{\ell_q}}{\chi_q\rho_q}\Big)\Big\|_{L^\infty_tH^N}\nonumber\\
\lesssim& \Big(1+\Big\|\Big(\frac{\RR_{\ell_q}}{\chi_q\rho_q}\Big)\Big\|_{L^\infty_{t,x}}\Big)^{N-1}\Big\|\Big(\frac{\RR_{\ell_q}}{\chi_q\rho_q}\Big)\Big\|_{L^\infty_tH^N}\Big\|\partial_t\Big(\frac{\RR_{\ell_q}}{\chi_q \rho_q}\Big)\Big\|_{L^\infty_tH^N}\nonumber\\
\le&
C_N (TK^3)^{N+1}\delta^{-3N-2}_{q+1}\lambda^{15N+10+4(2N+1)\alpha}_q\ell^{-4N}_q.
\label{t-ak-HN}
\end{align}
Collecting \eqref{chi-H4}, \eqref{t-chi-H4}, \eqref{ak-HN} and \eqref{t-ak-HN} together shows that
\begin{align*}
\|\partial_t {a}_{(k,q)}\|_{L^\infty_t H^N}\le& \|\partial_t\eta_q\|_{L^\infty_t H^N}\Big\| a_{k}\Big({\rm Id}-\frac{\RR_{\ell_q}}{\chi_q\rho_q}\Big)\Big\|_{L^\infty_t H^N}\|(\chi_q\rho_q)^{1/2}\|_{L^\infty_t H^N}\\
&+\|\eta_q\|_{L^\infty_t H^N}\Big\|\partial_t a_{k}\Big({\rm Id}-\frac{\RR_{\ell_q}}{\chi_q\rho_q}\Big)\Big\|_{L^\infty_t H^N}\|(\chi_q\rho_q)^{1/2}\|_{L^\infty_t H^N}\\
&+\|\eta_q\|_{L^\infty_t H^N}\Big\| a_{k}\Big({\rm Id}-\frac{\RR_{\ell_q}}{\chi_q\rho_q}\Big)\Big\|_{L^\infty_t H^N}\|\partial_t(\chi_q\rho_q)^{1/2}\|_{L^\infty_t H^N}\\
\le& C_NT^{-1/2}(TK^3)^{N+1}\delta^{-4N-3/2}_{q+1}\lambda^{21N+10+4(3N+1)\alpha}_q\ell^{-5N}_q,
\end{align*}
where we use the fact that $\|\partial_t\eta_q\|_{L^\infty_t H^N}\lesssim K^{3/2}T^{-1}\lambda^N_{q}$ and $\|\eta_q\|_{L^\infty_t H^N}\lesssim K^{3/2}T^{-1}\lambda^N_{q}$. Thanks to \eqref{Con1}, we prove \eqref{t-akq-HN}. Thus, we complete the proof of Proposition \ref{est-ak}.
\end{proof}

\begin{proposition}[Estimates for $F_{q+1}$]\label{est-Fq-B}Let $F_{q+1}$ be defined in \eqref{def-F_{q+1}}, then
    \begin{align}
        &\Big\|\int_{0}^T e^{(t-s)\Delta}F_{q+1}(s)\dd s\Big\|_{\widetilde L^\infty_tB^{1/2}_{2,1}\cap \widetilde  L^1_tB^{5/2}_{2,1}}
        \le  \lambda^{-20}_q, \label{heat-F}\\
        &{\|F_{q+1}\|_{ L^\infty_t H^2}\lesssim N_{\Lambda}K^{10}\ell^{-12}_q\lambda^{3+8\epsilon_0}_{q+1}. } \label{es-F-H1}
    \end{align}
\end{proposition}
\begin{proof}{We firstly prove \eqref{heat-F}. By Lemma \ref{heat}, we have for $T<1$ that
\begin{align*}
&\Big\|\int_{0}^T e^{(t-s)\Delta}F^{(1)}_{q+1}(s)\dd s\Big\|_{_{\widetilde L^\infty_tB^{1/2}_{2,1}\cap \widetilde  L^1_tB^{5/2}_{2,1}}}\lesssim\|F^{(1)}_{q+1}\|_{ L^1_tB^{1/2}_{2,1}}.\end{align*}
The fact that $\supp_x a_{(k,q)}\subseteq \Omega_{q+1}\subseteq \big[-\tfrac{K}{2}, \tfrac{K}{2}\big]^3$ implies that $F^{(1)}_{q+1}$ has spatial compact support. Therefore, we obtain by the definition of $\TTT^3$ that
\begin{align*}
 \|F^{(1)}_{q+1}\|_{ L^1_tB^{1/2}_{2,1}}\lesssim&\|F^{(1)}_{q+1}\|^{1/2}_{ L^1_tL^2}  \|F^{(1)}_{q+1}\|^{1/2}_{ L^1_tH^1} \\
\lesssim&\lambda^{-2}_{q+1}\|\psi_{\lambda^{-\epsilon_0}_{q+1}} (\lambda^{1-\epsilon_0}_{q+1}N_{\Lambda}k\cdot x)\|^{1/2}_{L^2(\TTT^3)} \|\psi_{\lambda^{-\epsilon_0}_{q+1}} (\lambda^{1-\epsilon_0}_{q+1}N_{\Lambda}k\cdot x)\|^{1/2}_{H^1(\TTT^3)}\\
&\times\| {\partial_t({a}_{(k,q)} g_{(2, \sigma_0, k)}(t))}\|_{L^1_tW^{3,\infty}}.
\end{align*}
By Proposition \ref{guji1} and Proposition \ref{est-ak}, we obtain that
\begin{align*}
&\| {\partial_t({a}_{(k,q)} g_{(2, \sigma_0, k)}(t))}\|_{L^1_t W^{3,\infty}}\\
\lesssim&\|{a}_{(k,q)}\|_{L^\infty_t W^{3,\infty}}\|\partial_t g_{(2, \sigma_0, k)}(t)\|_{L^1_t}+\|\partial_t{a}_{(k,q)}\|_{L^\infty_t  W^{3,\infty}}\| g_{(2, \sigma_0, k)}(t)\|_{L^1_t}\\
\lesssim& \ell^{-12}_q\lambda^{1+\sigma_0}_{q+1},
\end{align*}
\begin{align*}
&\|\psi_{\lambda^{-\epsilon_0}_{q+1}} (\lambda^{1-\epsilon_0}_{q+1}N_{\Lambda}k\cdot x)\|^{1/2}_{L^2(\TTT^3)}\lesssim N^{1/2}_{\Lambda}K^{3/4},
\end{align*}
 and
\begin{align*}
\|\psi_{\lambda^{-\epsilon_0}_{q+1}} (\lambda^{1-\epsilon_0}_{q+1}N_{\Lambda}k\cdot x)\|^{1/2}_{H^1(\TTT^3)}\lesssim N^{1/2}_{\Lambda}K^{3/4}\lambda^{1/2}_{q+1}.
\end{align*}
Hence, one gets
\begin{align}
&\Big\|\int_{0}^T e^{(t-s)\Delta}F^{(1)}_{q+1}(s)\dd s\Big\|_{_{_{\widetilde L^\infty_tB^{1/2}_{2,1}\cap \widetilde  L^1_tB^{5/2}_{2,1}}}}\lesssim N_{\Lambda}K^{3/2}\ell^{-12}_q\lambda^{\sigma_0-1/2}_{q+1}.\label{es:F1}
\end{align}
By the definition of $F^{(2)}_{q+1}$ in Proposition \ref{def-F2}, we have
\begin{align*}
\|F^{(2)}_{q+1}\|_{L^1_tB^{1/2}_{2,1}}\lesssim& \|F^{(2)}_{q+1}\|^{1/2}_{L^1_tL^2}\|F^{(2)}_{q+1}\|^{1/2}_{L^1_tH^1}\\
\lesssim& \sum_{k\in\Lambda}\sum_{m,l\in \ZZ\backslash\{0\}, m+l\neq 0} K\Big\|{a}^2_{(k,q)}g^2_{(2, \sigma_0, k)} \|_{L^1_tW^{3,\infty}}\\
&\times\Big\|
\frac{c_{l, \epsilon_0}c_{m, \epsilon_0}e^{ \ii2\pi \lambda^{1-\epsilon_0}_{q+1} (l+m)N_{\Lambda} k\cdot x/K }}{2\pi N_{\Lambda}\lambda^{1-\epsilon_0}_{q+1}(l+m)} \Big\|^{1/2}_{L^2(\TTT^3)}\Big\|
\frac{c_{l, \epsilon_0}c_{m, \epsilon_0}e^{ \ii2\pi \lambda^{1-\epsilon_0}_{q+1} (l+m)N_{\Lambda} k\cdot x/K }}{2\pi N_{\Lambda}\lambda^{1-\epsilon_0}_{q+1}(l+m)} \Big\|^{1/2}_{H^1(\TTT^3)}.
\end{align*}
Note that
\begin{align*}
\|e^{ \ii2\pi \lambda^{1-\epsilon_0}_{q+1} (l+m)N_{\Lambda} k\cdot x/K }\|_{H^1(\TTT^3)}\lesssim K^{1/2}N_{\Lambda}\lambda^{1-\epsilon_0}_{q+1}|l+m|,
\end{align*}
together with Lemma \ref{heat}, Proposition \ref{est-cmr}-- \ref{guji1} and Proposition \ref{est-ak}, we have
\begin{align}
&\Big\|\int_{0}^T e^{(t-s)\Delta}F^{(2)}_{q+1}(s)\dd s\Big\|_{\widetilde L^\infty_tB^{1/2}_{2,1}\cap \widetilde  L^1_tB^{5/2}_{2,1}}\lesssim\|F^{(2)}_{q+1}\|_{L^1_tB^{1/2}_{2,1}}\nonumber\\
\lesssim&\sum_{k\in\Lambda}\sum_{m,l\in \ZZ\backslash\{0\}, m+l\neq 0} K^{2}|c_{l, \epsilon_0}||c_{m, \epsilon_0}||l+m|^{-1/2} \|{a}^2_{(k,q)} \|_{L^\infty_tW^{3,\infty}}\lambda^{-(1-\epsilon_0)/2}_{q+1}\nonumber\\
\lesssim&K^{11}\lambda^{5+4\alpha}_{q}\ell^{-10}_q\lambda^{-1/2+{10\epsilon_0}}_{q+1}.\label{es:F2}
\end{align}
Now we estimate $\int_0^t e^{(t-s)\Delta} F^{(3)}_{q+1}(s)\dd s$. Firstly,  we denote
\begin{align}\label{def-h-sig}
h_{\sigma_0}(t)=\lambda^{-\sigma_0}_{q+1}\int_0^{\lambda^{\sigma_0}_{q+1}t}\mathbb{P}_{\neq0}(g^2_{(2, \sigma_0, k)}(s))\dd s.
\end{align}
Then we have
\begin{align*}
\int_0^t e^{(t-s)\Delta}F^{(3)}_{q+1}(s)\dd s=&\int_0^t e^{(t-s)\Delta}\Div(a^2_{(k,q)}\bar{k}\otimes \bar{k})\dd h_{\sigma_0}(s)\\
=&\Div(a^2_{(k,q)}(t)\bar{k}\otimes \bar{k})h_{\sigma_0}(t)-\int_0^th_{\sigma_0}(s)\partial_s(e^{(t-s)\Delta}\Div(a^2_{(k,q)}\bar{k}\otimes\bar{k}))\dd s.
\end{align*}
Since
\begin{align*}
&\int_0^th_{\sigma_0}(s)\partial_s(e^{(t-s)\Delta}\Div(a^2_{(k,q)}\bar{k}\otimes\bar{k}))\dd s\\
=&\int_0^th_{\sigma_0}(s)e^{(t-s)\Delta}\partial_s(\Div(a^2_{(k,q)}\bar{k}\otimes\bar{k}))\dd s+\int_0^th_{\sigma_0}(s)\partial_s(e^{(t-s)\Delta})\Div(a^2_{(k,q)}\bar{k}\otimes\bar{k}))\dd s\\
=&\int_0^th_{\sigma_0}(s)e^{(t-s)\Delta}\partial_s(\Div(a^2_{(k,q)}\bar{k}\otimes\bar{k}))\dd s-\int_0^th_{\sigma_0}(s)\Delta(e^{(t-s)\Delta}\Div(a^2_{(k,q)}\bar{k}\otimes\bar{k}))\dd s,
\end{align*}
we obtain by Proposition \ref{est-ak} that
\begin{align}
\Big\|\int_0^t e^{(t-s)\Delta}F^{(3)}_{q+1}(s)\dd s\Big\|_{\widetilde L^\infty_t B^{1/2}_{2,1}\cap \widetilde L^1_t B^{5/2}_{2,1}}
\lesssim& (\|a^2_{(k,q)}\|_{L^\infty_t  H^4}+\|\partial_t(a^2_{(k,q)})\|_{L^\infty_t H^2})\| h_{\sigma_0}(t)\|_{L^\infty_t}\nonumber\\
\lesssim& \lambda^{5+4\alpha}_{q}\ell^{-16}_q\lambda^{-\sigma_0}_{q+1}.\label{es:F3}
\end{align}
Moreover, by   \eqref{e:vq-H3} and \eqref{e:R_rem}, we have
\begin{align*}
\|(1-\zeta_q)\Div \RR^{rem}_q\|_{\LoB}\lesssim\| \RR^{rem}_q\|_{L^\infty_tH^2}\lesssim\lambda^{10}_q\ell_q\lesssim\lambda^{-40}_q.
\end{align*}
and
\begin{align*}
    &\|\partial_t\zeta_q(u_q-u_{\ell_q})+\zeta_q(1-\zeta_q)\Div ((u_q-u_{\ell_q})\otimes (u_q-u_{\ell_q}))\|_{\LoB}\\ \lesssim&\ell_q\|u_q\|_{L^\infty_TH^2}+\ell_q\|u_q\|_{L^\infty_tH^3}(\|u_q\|_{L^\infty_tH^2}+\|u_{\ell_q}\|_{L^\infty_tH^2})\\
  \lesssim&\lambda^{10}_q\ell_q\lesssim\lambda^{-40}_q.
\end{align*}
Collecting the above two estimates with \eqref{es:F1}--\eqref{es:F3} together shows that
\begin{align*}
&\Big\|\int_{0}^T e^{(t-s)\Delta}F_{q+1}(s)\dd s\Big\|_{\widetilde L^\infty_tB^{1/2
}_{2,1}\cap \LoBt}\\
\lesssim&N_{\Lambda}K^{3/2}\ell^{-12}_q\lambda^{-1/2+\sigma_0}_{q+1} +K^{11}\lambda^{5+4\alpha}_{q}\ell^{-10}_q\lambda^{-1/2+{10\epsilon_0}}_{q+1}+\lambda^{5+4\alpha}_{q}\ell^{-16}_q\lambda^{-\sigma_0}_{q+1}+\lambda^{-40}_q\\
\lesssim&K^{11}N_{\Lambda}\lambda^{-40}_q\le \lambda^{-20}_q.
\end{align*}
Thereby, we obtain \eqref{heat-F} for large enough a.

Now we turn to estimate $\|F_{q+1}\|_{ L^\infty_t H^2}$. Firstly, we estimate $\|F^{(1)}_{q+1}\|_{ L^\infty_t H^2}$, $\|F^{(2)}_{q+1}\|_{ L^\infty_t H^2}$ and $\|F^{(3)}_{q+1}\|_{ L^\infty_t H^2}$ respectively.}

For $F^{(1)}_{q+1}$ defined in Proposition \ref{F1}, using Proposition \ref{guji1} and Proposition \ref{est-ak}, we have
   \begin{align*}
    \|F^{(1)}_{q+1}\|_{L^\infty_t H^2}\lesssim & \lambda^{-2}_{q+1}\|\psi_{\lambda^{-\epsilon_0}_{q+1}} (\lambda^{1-\epsilon_0}_{q+1}N_{\Lambda}k\cdot x)\|_{H^2(\TTT^3)}\|\partial_t(a_{(k,q)}  g_{(2, \sigma_0, k)}(t))\|_{L^\infty_tW^{4,\infty}}\\
\lesssim&K^{3/2}N^2_{\Lambda}\ell^{-12}_q\lambda^{3+\sigma_0}_{q+1}.
   \end{align*}
 Owing to \eqref{cm-2}, Proposition \ref{guji1} and Proposition \ref{est-ak}, we  bound $F^{(2)}_{q+1}$ by
   \begin{align*}
 &\|F^{(2)}_{q+1}\|_{L^\infty_t H^2}\\
\lesssim &K\sum_{k\in\Lambda}\sum_{m,l \in\ZZ\backslash\{0\}, m+l\neq 0} \frac{|c_{l,\epsilon_0}||c_{m,\epsilon_0}|}{\lambda^{1-\epsilon_0}_{q+1}|l+m|}\|e^{\ii2\pi \lambda^{1-\epsilon_0}_{q+1}(l+m)N_{\Lambda }k\cdot x/K}\|_{W^{2,\infty}(\TTT^3)}\|a^2_{(k,q)} g^2_{(2, \sigma_0,k)}\|_{L^\infty_tH^4}\\
\lesssim&N_{\Lambda}K^{10}\sum_{m,l \in\ZZ\backslash\{0\}, m+l\neq 0} \lambda^{1+8\epsilon_0}_{q+1}|l+m|^{-3}\|a^2_{(k,q)} \|_{L^\infty_t H^4}\|g^2_{(2, \sigma_0,k)}\|_{L^\infty_t}\\
\lesssim&N_{\Lambda}K^{10}\ell^{-8}_q\lambda^{3+8\epsilon_0}_{q+1}.
 \end{align*}
 For  $F^{(3)}_{q+1}$,   by Proposition \ref{guji1} and Proposition \ref{est-ak}, we have
  \begin{align*}
\|F^{(3)}_{q+1}\|_{L^\infty_t H^2}\le& \|a^2_{(k,q)}\|_{L^\infty_t H^2}\|\mathbb{P}_{\neq0}(g^2_{(2, \sigma_0, k)}) \|_{L^\infty_t}
\lesssim\lambda^{5+4\alpha}_{q}\ell^{-4}_q\lambda^{2}_{q+1}.
    \end{align*}
Using \eqref{e:vq-H3} and \eqref{e:R_rem}, one deduces that
 \begin{align*}
    \|\Div \RR^{rem}_q\|_{L^\infty_t H^2}\lesssim\lambda^{10}_{q}.
 \end{align*}
 and
 \begin{align*}
    &\|\partial_t\zeta_q(u_q-u_{\ell_q})+\zeta_q(1-\zeta_q)\Div ((u_q-u_{\ell_q})\otimes (u_q-u_{\ell_q}))\|_{{ L^\infty_t H^2}}\\
\lesssim &\ell_q\|u_q\|_{L^\infty_tH^3}+(\|u_q\|_{L^\infty_tH^3}+\|u_{\ell_q}\|_{L^\infty_tH^3})^2
    \lesssim\lambda^{10}_q.
\end{align*}
Therefore, we have
 \begin{align*}
 \|F_{q+1}\|_{ L^\infty_tH^2}\lesssim& K^{3/2}N^2_{\Lambda}\ell^{-12}_q\lambda^{3+\sigma_0}_{q+1}+N_{\Lambda}K^{10}\ell^{-8}_q\lambda^{3+8\epsilon_0}_{q+1}+\lambda^{5+4\alpha}_{q}\ell^{-4}_q\lambda^{2}_{q+1}+C\lambda^{10}_q\\
\lesssim&N_{\Lambda}K^{10}\ell^{-12}_q\lambda^{3+8\epsilon_0}_{q+1}.
 \end{align*}
 We obtain \eqref{es-F-H1}, so that we finish the proof of Proposition \ref{est-Fq-B}.
\end{proof}
   \begin{proposition}[Estimates for $\wtq$]\label{wtq-H3}Let $\wtq$ be the solution of the equations \eqref{e:wt}, then we have
\begin{align}
&\|\wtq\|_{\widetilde L^{\infty}_tB^{1/2}_{2,1}\cap \widetilde L^{1}_tB^{5/2}_{2,1} }\le \frac{1}{2}\delta_{q+2}\lambda^{-6\alpha}_{q+1},\label{estimate-wt}\\
&{\|\wtq\|_{\widetilde L^{\infty}_tB^3_{2,2}}\lesssim N_{\Lambda}K^{10}\ell^{-12}_q\lambda^{3+8\epsilon_0}_{q+1}} .
\end{align}
   \end{proposition}
   \begin{proof}
       With the aid of Duhamel formula, we write the equations \eqref{e:wt} in the integral form
\begin{align}
\wtq(x,t)=&\int_0^t e^{(t-s)\Delta}\mathcal{P}\Div (\wtq\otimes \wtq+\uqnl\otimes \wtq+ \wtq\otimes\uqnl)(s) \dd s\nonumber\\
&-\int_0^t e^{(t-s)\Delta}\mathcal{P} F_{q+1}(s)\dd s,\label{inte-wtq}
\end{align}
where $\mathcal{P}$ is the Leray projector onto divergence-free victor fields.

Thanks to \eqref{e:vqnl-H3}, we have
\begin{align*}
&\Big\|\int_0^t e^{(t-s)\Delta}\mathcal{P}\Div ( \uqnl\otimes \wtq+ \wtq\otimes\uqnl )\dd s\Big\|_{\widetilde L^{\infty}_tB^{1/2}_{2,1}\cap \widetilde L^{1}_tB^{5/2}_{2,1} }\\
\lesssim&(\|\uqnl\otimes \wtq\|_{\widetilde L^{1}_tB^{3/2}_{2,1}}+\|\wtq\otimes \uqnl\|_{\widetilde L^{1}_tB^{3/2}_{2,1}})\\
\lesssim&\|\uqnl\|_{\widetilde{L}^{\infty}_tB^{1/2}_{2,1}\cap \widetilde{L}^{1}_tB^{5/2}_{2,1} }\|\wtq\|_{\widetilde L^{\infty}_tB^{1/2}_{2,1}\cap \widetilde L^{1}_tB^{5/2}_{2,1} }\\
\le&CM^{-1}\|\wtq\|_{\widetilde L^{\infty}_tB^{1/2}_{2,1}\cap \widetilde L^{1}_tB^{5/2}_{2,1} }.
\end{align*}
By Proposition \ref{est-Fq-B}, one gets
\begin{align*}
&\Big\|\int_0^t e^{(t-s)\Delta}\mathcal{P} F_{q+1}(s)\dd s\Big\|_{\widetilde L^{\infty}_tB^{1/2}_{2,1}\cap \widetilde  L^{1}_tB^{5/2}_{2,1} }\\
\lesssim&\Big\|\int_0^t e^{(t-s)\Delta} F_{q+1}(s)\dd s\Big\|_{\widetilde L^{\infty}_tB^{1/2}_{2,1}\cap \widetilde L^{1}_tB^{5/2}_{2,1} }
\lesssim\lambda^{-20}_{q}\le \delta_{q+2}\lambda^{-8\alpha}_{q+1}.
\end{align*}
Taking $M$ large enough such that $CM^{-1}\le \frac{1}{2}$, for large enough $a$, we collect the above two estimates together to obtain
\begin{align*}
\|\wtq\|_{\widetilde L^{\infty}_tB^{1/2}_{2,1}\cap \widetilde      L^{1}_tB^{5/2}_{2,1} }\lesssim\|\wtq\|^2_{ \widetilde L^{\infty}_tB^{1/2}_{2,1}\cap \widetilde L^{1}_tB^{5/2}_{2,1} }+\lambda^{-20}_{q}\le \delta_{q+2}\lambda^{-8\alpha}_{q+1}.
\end{align*}
By  the Banach fixed point theorem and the the continuity method, this estimate implies that as long as $a$ is large enough, the equation \eqref{inte-wtq} admits a unique  mild solution $\wtq$ on $[0,T]$ with
\begin{align}
 \|\wtq\|_{ \widetilde L^{\infty}_tB^{1/2}_{2,1}\cap \widetilde L^{1}_tB^{5/2}_{2,1} }\lesssim\delta_{q+2}\lambda^{-8\alpha}_{q+1}\le \frac{1}{2} \delta_{q+2}\lambda^{-6\alpha}_{q+1}.\label{es：wtq}
\end{align}
Note that $F_{q+1}(t)=0$ for $t\in [0, \tfrac{T}{4}+2\lambda^{-1}_{q-1}]$ and $\wtq|_{t=0}=0$, the above inequality implies that
\begin{align}\label{supp-wtq}
    \wtq(t)=0, \quad \forall \,\,0\le t\le \tfrac{T}{4}+2\lambda^{-1}_{q-1}.
\end{align}
Moreover, by \eqref{e:vqnl-H3}, \eqref{es-F-H1} and  \eqref{es：wtq}, note that $L^\infty_tH^2\hookrightarrow \widetilde L^{\infty}_t B^1_{2,2}$, we have
\begin{align*}
 \|\wtq\|_{ \widetilde L^{\infty}_t B^3_{2,2} }
\lesssim&\|\wtq\|_{\widetilde L^{\infty}_tB^{-1}_{\infty,\infty}}\|\wtq\|_{\widetilde L^{\infty}_t B^3_{2,2}}+\|\uqnl\|_{L^\infty_t L^\infty}\|\wtq\|_{\widetilde L^{\infty}_t B^2_{2,2}}\\
 &+\|\uqnl\|_{\widetilde L^{\infty}_t B^2_{2,2}}\|\wtq\|_{L^\infty_t L^\infty}+\|F_{q+1}\|_{\widetilde L^{\infty}_t B^1_{2,2}} \\
 \lesssim&\|\wtq\|_{\widetilde L^{\infty}_tB^{1/2}_{2,1}}\|\wtq\|_{\widetilde L^{\infty}_t B^3_{2,2}}+\|\uqnl\|_{L^\infty_t L^\infty}\|\wtq\|^{3/5}_{\widetilde L^{\infty}_t B^3_{2,2}}\|\wtq\|^{2/5}_{\widetilde L^{\infty}_t B^{1/2}_{2,1}}\\
& +\|\uqnl\|_{\widetilde L^{\infty}_t B^2_{2,2}}\|\wtq\|^{2/5}_{\widetilde L^{\infty}_t B^3_{2,2}}\|\wtq\|^{3/5}_{\widetilde L^{\infty}_t B^{1/2}_{2,1}}+\|F_{q+1}\|_{\widetilde L^{\infty}_t B^1_{2,2}}\\
\le&C(\delta_{q+2}\lambda^{-6\alpha}_{q+1}\|\wtq\|_{\widetilde L^{\infty}_t B^3_{2,2}}+\delta_{q+2}\lambda^{-6\alpha}_{q+1}\lambda^{20}_q+N_{\Lambda}K^{10}\ell^{-12}_q\lambda^{3+8\epsilon_0}_{q+1})+\frac{1}{2}\|\wtq\|_{ \widetilde L^{\infty}_t B^3_{2,2} }.
\end{align*}
This inequality shows that
\begin{align*}
 \|\wtq\|_{ L^{\infty}_t H^3 } \le \|\wtq\|_{\widetilde L^{\infty}_t B^3_{2,2} }\lesssim N_{\Lambda}K^{10}\ell^{-12}_q\lambda^{3+8\epsilon_0}_{q+1}.
\end{align*}
   \end{proof}
   \begin{proposition}[Estimates for $w_{q+1}$]\label{estimate-wq+1}For $1\le p<2$, there exists a  constant {$C_0$} such that
\begin{align}
&\|w^{(p)}_{q+1}\|_{L^2_{t,x}\cap L^pL^{\infty}}+{\lambda^{-9/2}_{q+1}}\|w^{(p)}_{q+1}\|_{L^{\infty}H^3}\le \frac{1}{4}C_0\delta^{1/2}_{q+1},\label{estimate-wp}\\
&\|w^{(c)}_{q+1}\|_{L^2_{t,x}\cap L^pL^{\infty}}+{\lambda^{-4}_{q+1}}\|w^{(c)}_{q+1}\|_{L^{\infty}H^3}
\le \lambda^{-1+\frac{\epsilon_0}{2}}_{q+1},\label{estimate-wc}\\
&\|w_{q+1}\|_{L^2_{t,x}\cap L^pL^{\infty}}+{\lambda^{-9/2}_{q+1}}\|w_{q+1}\|_{L^{\infty}_tH^3}\le \frac{1}{2}C_0\delta^{1/2}_{q+1}.\label{estimate-w}
\end{align}
Moreover,
\begin{align}
  \|\wpq\|_{L^\infty_tL^2}\lesssim\ell^{-4}_q\lambda_{q+1},  \quad\|\wcq\|_{L^\infty_tL^2}\lesssim\ell^{-6}_q. \label{w-LinftyL2}
\end{align}
   \end{proposition}
\begin{proof}By the definition of $\wpq$ in \eqref{def-wpq-1} and Lemma \ref{Holder}, we have
  \begin{align*}
\|w^{(p)}_{q+1}\|_{L^2_{t,x}}
\lesssim& \|a_{(k,q)}\|_{L^2_{t,x}}\|g_{(2, \sigma_0, k)}(t)\|_{L^2_t}\|\phi_{(\epsilon_0, 1-\epsilon_0, k)}(x) \|_{L^2_x}\\
&+\lambda^{-\frac{\sigma_0}{2}-\frac{1-\epsilon_0}{2}}_{q+1}\|a_{(k,q)}\|_{C^1_{t,x}}\|g_{(2, \sigma_0, k)}(t)\|_{L^2_t}\|\phi_{(\epsilon_0, 1-\epsilon_0, k)}(x) \|_{L^2_x}\\
\lesssim&K^{2}\|a_{(k,q)}\|_{L^2_{t,x}}+K^{3/2}\lambda^{-\frac{\sigma_0}{2}}_{q+1}\|a_{(k,q)}\|_{C^1_{t,x}}.
\end{align*}
From the definition of $a_{(k,q)}$, one deduces that
 \begin{align*}
\|a_{(k,q)}\|_{L^2_{t,x}}\lesssim\|\chi_q\rho_q\|^{1/2}_{L^1_{t,x}}\lesssim\delta^{1/2}_{q+1}.
 \end{align*}
By Proposition \ref{est-ak}, we have
 \begin{align*}
 \|a_{(k,q)}\|_{C^{1}_{t,x}}\lesssim\ell^{-12}_q.
 \end{align*}
 Hence,  there exists a universal constant $C_0>K^{2}$ such that
 \begin{align*}
\|w^{(p)}_{q+1}\|_{L^2_{t,x}}\lesssim& K^{2}\delta^{1/2}_{q+1}+K^{3/2}\ell^{-12}_q\lambda^{-\frac{\sigma_0}{2}-\frac{1-\epsilon_0}{2}}_{q+1}
\le \frac{1}{8}C_0\delta^{1/2}_{q+1}.
 \end{align*}
 Thanks to Proposition \ref{guji1}, one gets
 \begin{align}\label{est-wp-Lp}
\|w^{(p)}_{q+1}\|_{ L^pL^{\infty}}
\le& \|a_{(k,q)}\|_{L^\infty_{t,x}}\|g_{(2, \sigma_0, k)}(t)\|_{L^p}\|\phi_{(\epsilon_0, 1-\epsilon_0, k)}(x) \|_{L^\infty(\TTT^3)}\nonumber\\
\lesssim &K^{3/2}\lambda^{5+4\alpha}_{q}\lambda^{-2(\frac{1}{p}-\frac{1}{2})+\frac{\epsilon_0}{2}}_{q+1}\le \frac{1}{8}C_0\delta^{1/2}_{q+1}.
\end{align}
We infer from Proposition \ref{est-ak} that
\begin{align*}
\|w^{(p)}_{q+1}\|_{L^\infty_tH^3}
\le& \|a_{(k,q)}\|_{L^\infty_tH^3}\|g_{(2, \sigma_0, k)}(t)\|_{L^\infty_t}\|\phi_{(\epsilon_0, 1-\epsilon_0, k)}(x) \|_{H^3(\TTT^3)}\\
\lesssim&N^3_{\Lambda}K^{3/2}\ell^{-6}_q\lambda^{4}_{q+1}\le \frac{1}{8}C_0\lambda^{9/2}_{q+1}\delta^{1/2}_{q+1}.
\end{align*}
Collecting the above three estimates together shows \eqref{estimate-wp}.

Applying Lemma \ref{Holder} and Proposition \ref{guji1} to $\wcq$ in \eqref{def-wcq}, we have
\begin{align*}
\|w^{(c)}_{q+1}\|_{L^2_{t,x}}
\lesssim &\lambda^{-1}_{q+1} \|a_{(k,q)}\|_{L^2_tH^1}\|g_{(2, \sigma_0, k)}(t)\|_{L^2_t}\|\psi'_{\lambda^{-\epsilon_0}_{q+1}}(\lambda^{1-\epsilon_0}_{q+1}N_{\Lambda}k\cdot x)  \|_{L^2(\TTT^3)}\\
&+\lambda^{-\frac{\sigma_0}{2}-\frac{1-\epsilon_0}{2}-1}_{q+1}\|a_{(k,q)}\|_{C^{1}_{t,x}}\|g_{(2, \sigma_0, k)}(t)\|_{L^2_t}\|\psi'_{\lambda^{-\epsilon_0}_{q+1}}(\lambda^{1-\epsilon_0}_{q+1}N_{\Lambda}k\cdot x) \|_{L^2(\TTT^3)}\\
\lesssim&K^{3/2}\lambda^{-1}_{q+1} \|a_{(k,q)}\|_{L^\infty_tH^1}+K^{3/2}\lambda^{-\frac{\sigma_0}{2}-\frac{1-\epsilon_0}{2}-1}_{q+1}\|a_{(k,q)}\|_{C^{1}_{t,x}}.
\end{align*}
Then utilizing Proposition \ref{est-ak} yields that
\begin{align*}
\|w^{(c)}_{q+1}\|_{L^2_{t,x}}
\lesssim K^{3/2}\ell^{-2}_q\lambda^{-1}_{q+1}.
\end{align*}
By Proposition \ref{guji1} and Proposition \ref{est-ak}, we have
\begin{align}
\|w^{(c)}_{q+1}\|_{L^pL^{\infty}}
\le &\lambda^{-1}_{q+1}\|a_{(k,q)}\|_{L^\infty_{t,x}}\|g_{(2, \sigma_0, k)}(t)\|_{L^p}\|\psi'_{\lambda^{-\epsilon_0}_{q+1}}(\lambda^{1-\epsilon_0}_{q+1}N_{\Lambda}k\cdot x)\|_{L^\infty(\TTT^3)}\nonumber\\
\lesssim&K^{1/2}\lambda^{5+4\alpha}_{q}\lambda^{-1-2(\frac{1}{p}-\frac{1}{2})+\frac{\epsilon_0}{2}}_{q+1}\label{est-wcq-Lp}
\end{align}
and
\begin{align*}
\|w^{(c)}_{q+1}\|_{L^\infty_tH^3}\le &\lambda^{-1}_{q+1}\|a_{(k,q)}\|_{L^\infty_{t}H^3}\|g_{(2, \sigma_0, k)}(t)\|_{L^\infty}\|\psi'_{\lambda^{-\epsilon_0}_{q+1}}(\lambda^{1-\epsilon_0}_{q+1}N_{\Lambda}k\cdot x)\|_{H^3(\TTT^3)}\\
\lesssim&N^3_{\Lambda}K^{3/2}\ell^{-6}_q\lambda^3_{q+1}.
\end{align*}
Hence, we obtain \eqref{estimate-wc}. Due to
\begin{align*}
    L^\infty_t B^{1/2}_{2,1}\cap L^1_t B^{5/2}_{2,1} \hookrightarrow L^2_{t,x}\cap L^p_tL^\infty, \quad 1\le p<2,
\end{align*}
the estimates \eqref{estimate-wt}, \eqref{estimate-wp} together with \eqref{estimate-wc}  imply \eqref{estimate-w}. {Moreover, by  Proposition~\ref{guji1} and Proposition~\ref{est-ak}, we obtain that
\begin{align*}
    \|\wpq\|_{L^\infty_tL^2}
\lesssim& \|a_{(k,q)}\|_{L^\infty_tL^\infty}\|g_{(2, \sigma_0, k)}(t)\|_{L^\infty_t}\|\phi_{(\epsilon_0, 1-\epsilon_0, k)}(x) \|_{L^2(\TTT^3)}
\lesssim \ell^{-4}_q\lambda_{q+1},\end{align*}
and
\begin{align*}
    \|\wcq\|_{L^\infty_tL^2}
\lesssim& \lambda^{-1}_{q+1}\|a_{(k,q)}\|_{L^\infty_tW^{1,\infty}}\|g_{(2, \sigma_0, k)}(t)\|_{L^\infty_t}\|\psi'_{\lambda^{-\epsilon_0}_{q+1}}(\lambda^{1-\epsilon_0}_{q+1}N_{\Lambda}k\cdot x)  \|_{L^2(\TTT^3)}
\lesssim\ell^{-6}_q.\end{align*}}
This shows \eqref{w-LinftyL2}. Therefore, we complete the proof of Proposition \ref{estimate-wq+1}.
\end{proof}
\subsubsection{Estimates for the Reynolds stress $\RR_{q+1}$} Letting
\begin{align*}
&u_{q+1}=\bar{u}_q+w_{q+1}, \\
&p_{q+1}=p_t+\bar{p}_q-P^{(1)}_{q+1}-\eta^2_q (\chi_q \rho_q)^{1/2}+\frac{2}{3}(\wpq+\wcq)\cdot\vv_q+\frac{2}{3}\wtq\cdot\ubql\\
&\qquad\quad+{\frac{2}{3}\wpq\cdot\wcq+\frac{2}{3}\wpq\cdot\wtq+\frac{2}{3}\wtq\cdot\wcq+\frac{1}{3}\wcq\cdot\wcq},
\end{align*}
where $w_{q+1}=\wpq+\wcq+\wtq$, we have from \eqref{e:mollified-euler} that
\begin{align*}
&\del_tu_{q+1}-\Delta u_{q+1} +\Div(u_{q+1}\otimes u_{q+1})+\nabla p_{q+1}\\
=&\partial_t w_{q+1}-\Delta w_{q+1}+\Div(w_{q+1}\otimes \vv_q)+\Div(\vv_q\otimes w_{q+1})+\Div( w_{q+1}\otimes w_{q+1})\\
&+\nabla p_t+\Div \RR_{\ell_q} +(1-\zeta_q)\Div \RR^{rem}_q+\partial_t\zeta_q(u_q-u_{\ell_q})+\zeta_q(1-\zeta_q)\Div ((u_q-u_{\ell_q})\otimes (u_q-u_{\ell_q}))\\
=& \partial_t(\wpq+\wcq)-\Delta(\wpq+\wcq)+\Div((\wpq+\wcq)\ootimes \vv_q)+\Div( \vv_q\ootimes (\wpq+\wcq))\\
&+\Div(\wtq\ootimes \ubql)+\Div( \ubql\ootimes \wtq)+\Div \RR_{\ell_q}\\
&+\partial_t \wtq-\Delta \wtq+\Div(\wtq\otimes \ubqnl)+\Div( \ubqnl\otimes \wtq)+\nabla p_t-\nabla P^{(1)}_{q+1}\\
&+\Div( \wpq\otimes \wpq)-\nabla\big(\eta^2_q (\chi_q \rho_q)^{1/2}\big)+(1-\zeta_q)\Div \RR^{rem}_q+\partial_t\zeta_q(u_q-u_{\ell_q})\\
&+{\Div\big(\wpq\ootimes (\wcq+\wtq)+\wcq\ootimes w_{q+1}+\wtq\ootimes(\wpq+\wcq)\big)}\\
&+\zeta_q(1-\zeta_q)\Div ((u_q-u_{\ell_q})\otimes (u_q-u_{\ell_q})).
\end{align*}

Since $\wtq$ satisfies the equations \eqref{e:wt}, by Proposition \ref{F1}--\ref{def-F2}, we obtain from the above equality that
\begin{align}
&\del_tu_{q+1}-\Delta u_{q+1} +\Div(u_{q+1}\otimes u_{q+1})+\nabla p_{q+1}\nonumber\\
=& \Div \RR^{(1)}_{q+1}+\Div \RR^{(2)}_{q+1}-\Delta(\wpq+\wcq)+\Div((\wpq+\wcq)\ootimes \vv_q)\nonumber\\
&+\Div( \vv_q\ootimes (\wpq+\wcq))+\Div(\wtq\otimes \ubql)+\Div( \ubql\otimes \wtq)+\Div \RR^{(3)}_{q+1}\nonumber\\
&+{\Div\big(\wpq\ootimes (\wcq+\wtq)+\wcq\ootimes w_{q+1}+\wtq\ootimes(\wpq+\wcq)\big)}\nonumber\\
=&\Div \big(\RR^{(1)}_{q+1}+ \RR^{(2)}_{q+1}+ \RR^{(3)}_{q+1}+(\wpq+\wcq)\ootimes \vv_q+\vv_q\ootimes (\wpq+\wcq)\nonumber\\
&+\wtq\ootimes \ubql+ \ubql\ootimes \wtq- \big(\nabla(\wpq+\wcq)+(\nabla(\wpq+\wcq))^{\T}\big)\nonumber\\
&+{\wpq\ootimes (\wcq+\wtq)+\wcq\ootimes w_{q+1}+\wtq\ootimes(\wpq+\wcq)\big)}\nonumber\\
=:&\Div \RR_{q+1}.\label{def-R-q+1}
\end{align}
Now we are focused on estimating $\RR_{q+1}$. According to the definition of $\RR_{q+1}$ in \eqref{def-R-q+1}, $\supp_x\RR_{q+1}$ is determined by the supports of  $\RR^{(1)}_{q+1}, \RR^{(2)}_{q+1}, \RR^{(3)}_{q+1}, \wpq, \wcq$ and $\ubql$. From  $\RR^{(1)}_{q+1} $ and $\RR^{(2)}_{q+1}$ in Proposition \ref{F1}, $\RR^{(3)}_{q+1}$ in Proposition \ref{def-F2}, $\wpq$ in \eqref{def-wpq} and $\wcq$ in \eqref{def-wcq}, we infer that
\begin{align*}
\supp_x\RR^{(1)}_{q+1}, \RR^{(2)}_{q+1}, \RR^{(3)}_{q+1}, \wpq, \wcq=\supp_x a_{(k,q)}\subseteq\Omega_{q+1}.
\end{align*}
This relation together with \eqref{e:ubql} implies that
\begin{align}\label{suppx-R}
    \supp_x \RR_{q+1}\subseteq\Omega_{q+1}\subseteq  \big[-\tfrac{K}{2}, \tfrac{K}{2}\big]^3.
\end{align}
\begin{proposition}[Estimates for $\RR_{q+1}$]\label{R-q+1}Let $\RR_{q+1}$ be defined in \eqref{def-R-q+1}, it holds that
\begin{align}
    & \|\RR_{q+1}\|_{L^1_{t,x}}\le \delta_{q+2}\lambda^{-4\alpha}_{q+1},\label{R—L1}\\
    &\|\RR_{q+1}\|_{L^\infty W^{3,1}}\le \lambda^5_{q+1}. \label{R—W}
\end{align}
\end{proposition}
\begin{proof}By the definitions of $\RR^{(1)}_{q+1}$ and $\RR^{(2)}_{q+1}$ in Proposition \ref{F1},  we infer from Proposition \ref{guji1}, Proposition \ref{est-ak} and \eqref{suppx-R} that
\begin{align}
    \|\RR^{(1)}_{q+1}\|_{L^1_{t,x}}
    \lesssim&\lambda^{-1}_{q+1}\|\partial_t g_{(2, \sigma_0, k)}\|_{L^1}\|{a}_{(k,q)}\|_{L^\infty_{t,x}}\|\psi'_{\lambda^{-\epsilon_0}_{q+1}}(\lambda^{1-\epsilon_0}_{q+1}N_{\Lambda}k\cdot x)\|_{L^1(\TTT^3)}\nonumber\\
    &+\lambda^{-1}_{q+1}\| g_{(2, \sigma_0, k)}\|_{L^1}\|\partial_t {a}_{(k,q)}\|_{L^\infty_{t,x}}\|\psi'_{\lambda^{-\epsilon_0}_{q+1}}(\lambda^{1-\epsilon_0}_{q+1}N_{\Lambda}k\cdot x)\|_{L^1(\TTT^3)}\nonumber\\
\lesssim&K^{5/2}\ell^{-12}_q\lambda^{\sigma_0-\frac{\epsilon_0}{2}}_{q+1},\label{es-R1}
\end{align}
and
  \begin{align}
    \|\RR^{(2)}_{q+1}\|_{L^1_{t,x}}\lesssim& \lambda^{-2}_{q+1}\|\psi_{\lambda^{-\epsilon_0}_{q+1}}(\lambda^{1-\epsilon_0}_{q+1}N_{\Lambda}k\cdot x)\|_{L^1(\TTT^3)}\| {a}_{(k,q)}\|_{L^\infty_tW^{1,\infty}}\|\partial_t g_{(2,\sigma_0,k)}\|_{L^1}\nonumber\\
    &+\lambda^{-2}_{q+1}\|\psi_{\lambda^{-\epsilon_0}_{q+1}}(\lambda^{1-\epsilon_0}_{q+1}N_{\Lambda}k\cdot x)\|_{L^1(\TTT^3)}\|\partial_t {a}_{(k,q)}\|_{{L^\infty_tW^{1,\infty}}}\| g_{(2,\sigma_0,k)}\|_{L^1}\nonumber\\
\lesssim&K^{5/2}\ell^{-18}_q\lambda^{\sigma_0-\frac{\epsilon_0}{2}-1}_{q+1}.\label{es-R2}
\end{align}
By the definition of $\RR^{(3)}_{q+1}$ in Proposition \ref{def-F2}, using \eqref{cm-2}, Proposition \ref{guji1} and Proposition \ref{est-ak}, one gets that
\begin{align}
    \|R^{(3)}_{q+1}\|_{L^1_{t,x}}\lesssim& K\sum_{m,l \in\ZZ\backslash\{0\}, m+l\neq 0} \frac{|c_{l, \epsilon_0}||c_{m, \epsilon_0}|}{\lambda^{1-\epsilon_0}_{q+1}|l+m|}\|a^2_{(k,q)}\|_{L^\infty_tW^{1,\infty}}\|g^2_{(2,\sigma_0,k)}\|_{L^1_t}\nonumber\\
   \lesssim& K^{10}\sum_{m,l \in\ZZ\backslash\{0\}, m+l\neq 0} |l+m|^{-5}\lambda^{-1+10\epsilon_0}_{q+1} \lambda^{5+4\alpha}_q\ell^{-6}_q\nonumber\\
\lesssim&K^{10}\lambda^{-1+10\epsilon_0}_{q+1} \lambda^{5+4\alpha}_q\ell^{-6}_q.\label{es-R3}
\end{align}
With the aid of \eqref{wpq+wcq}, we have
\begin{align*}
\|\nabla(\wpq+\wcq)\|_{L^1_{t,x}}
\lesssim&\lambda^{-1}_{q+1}\|a_{(k,q)}  g_{(2, \sigma_0, k)}(t)\psi'_{\lambda^{-\epsilon_0}_{q+1}}(\lambda^{1-\epsilon_0}_{q+1}N_{\Lambda}k\cdot x)\|_{L^1_tW^{2,1}}\\
\lesssim&\lambda^{-1}_{q+1}\|g_{(2, \sigma_0, k)}\|_{L^1}\|a_{(k,q)}\|_{L^\infty_tW^{2,\infty}}\|\psi'_{\lambda^{-\epsilon_0}_{q+1}}(\lambda^{1-\epsilon_0}_{q+1}N_{\Lambda}k\cdot x)\|_{W^{2,1}(\TTT^3)}.
\end{align*}
By Proposition \ref{guji1} and Proposition \ref{est-ak}, one obtains that
\begin{align*}
\|\nabla(\wpq+\wcq)\|_{L^1_{t,x}}\lesssim&K^{5/2}\ell^{-8}_q\lambda^{-\frac{\epsilon_0}{2}}_{q+1}.
\end{align*}
Making use of \eqref{e:ubq}, \eqref{estimate-wc} and \eqref{est-wp-Lp}, we have
\begin{align}
&\|(\wpq+\wcq)\ootimes \vv_q+\vv_q\ootimes (\wpq+\wcq)\|_{L^1_{t,x}}\nonumber\\
\lesssim&K^{3/2}(\|\wpq\|_{L^1_tL^\infty}+\|\wcq\|_{L^1_tL^\infty})\|\vv_q\|_{L^\infty_{t}L^2}\nonumber\\
\lesssim&K^{3/2}M\lambda^{5+4\alpha}_q\lambda^{-1+\frac{\epsilon_0}{2}}_{q+1}.\label{es-R4}
\end{align}
By virtue of \eqref{e:ubql} and \eqref{estimate-wt}, we easily deduce that
\begin{align}
\|\wtq\ootimes \ubql + \ubql\ootimes \wtq\|_{L^1_{t,x}}
\lesssim&K^{3/2}\|\wtq\|_{L^1_tL^{\infty}}\|\ubql\|_{L^\infty_tL^2}\nonumber\\
\lesssim&K^{3/2}M\delta_{q+2}\lambda^{-6\alpha}_{q+1}.\label{es-R5}
\end{align}
{With the aid of Proposition \ref{wtq-H3} and Proposition \ref{estimate-wq+1}, we have
\begin{align}
&\|\wpq\ootimes (\wcq+\wtq)\|_{L^1_{t,x}}+\|\wcq\ootimes w_{q+1}\|_{L^1_{t,x}}+\|\wtq\ootimes(\wpq+\wcq)\|_{L^1_{t,x}}\nonumber\\
\lesssim&\frac{1}{2}C_0\delta^{1/2}_{q+1}\delta_{q+2}\lambda^{-6\alpha}_{q+1}+\frac{1}{2}C_0\delta^{1/2}_{q+1}\lambda^{-1+\frac{\epsilon_0}{2}}_{q+1}.\label{es-R6}
\end{align}}
Collecting \eqref{es-R1}--\eqref{es-R6} together, we arrive at
\begin{align*}
\|\RR_{q+1}\|_{L^1_{t,x}}\lesssim K^{10}\delta_{q+2}\lambda^{-6\alpha}_{q+1}.
\end{align*}
For large enough $a$, this estimate combined with \eqref{epsilon} and \eqref{b-beta} yields \eqref{R—L1}.

{Now we are focused on  estimating  $\|\RR_{q+1}\|_{L^\infty_t W^{3,1}}$. By Proposition \ref{guji1} and Proposition \ref{est-ak}, we have
    \begin{align}
    \|\RR^{(1)}_{q+1}\|_{L^\infty_t W^{3,1}}
   \lesssim&\lambda^{-1}_{q+1}\|\partial_t g_{(2, \sigma_0, k)}\|_{L^\infty_t}\|{a}_{(k,q)}\|_{L^\infty_t W^{3,\infty}}\|\psi'_{\lambda^{-\epsilon_0}_{q+1}}(\lambda^{1-\epsilon_0}_{q+1}N_{\Lambda}k\cdot x)\|_{W^{3,1}(\TTT^3)}\nonumber\\
    &+\lambda^{-1}_{q+1}\| g_{(2, \sigma_0, k)}\|_{L^\infty_t}\|\partial_t {a}_{(k,q)}\|_{L^\infty_t W^{3,\infty}}\|\psi'_{\lambda^{-\epsilon_0}_{q+1}}(\lambda^{1-\epsilon_0}_{q+1}N_{\Lambda}k\cdot x)\|_{W^{3,1}(\TTT^3)}\nonumber\\
    \lesssim&K^{5/2}\ell^{-10}_q \lambda^{5+\sigma_0-\frac{\epsilon_0}{2}}_{q+1}.\nonumber
\end{align}
Applying Proposition \ref{guji1} and Proposition \ref{est-ak} to  $\RR^{(2)}_{q+1}$, we have
 \begin{align}
    \|\RR^{(2)}_{q+1}\|_{L^\infty_t W^{3,1}}
    \lesssim&\lambda^{-2}_{q+1}\|\psi_{\lambda^{-\epsilon_0}_{q+1}}(\lambda^{1-\epsilon_0}_{q+1}N_{\Lambda}k\cdot x)\|_{W^{3,1}(\TTT^3)}\| {a}_{(k,q)}\|_{L^\infty_t W^{4,\infty}}\|\partial_t g_{(2,\sigma_0,k)}\|_{L^\infty}\nonumber\\
    &+\lambda^{-2}_{q+1}\|\psi_{\lambda^{-\epsilon_0}_{q+1}}(\lambda^{1-\epsilon_0}_{q+1}N_{\Lambda}k\cdot x)\|_{{W^{3,1}}(\TTT^3)}\| \partial_t {a}_{(k,q)}\|_{L^\infty_t{W^{4,\infty}}}\| g_{(2,\sigma_0,k)}\|_{L^\infty}\nonumber\\
\lesssim&K^{5/2}\ell^{-12}_q\lambda^{4+\sigma_0-\frac{\epsilon_0}{2}}_{q+1}.\nonumber
\end{align}
For $\RR^{(3)}_{q+1}$, with the aid of \eqref{cm-2}, Proposition \ref{guji1} and Proposition \ref{est-ak}, one shows that
\begin{align}
    \|\RR^{(3)}_{q+1}\|_{L^\infty_t W^{3,1}}\lesssim& K^3N^3_{\Lambda}\sum_{m,l \in\ZZ\backslash\{0\}, m+l\neq 0}\lambda^{4-2\epsilon_0}_{q+1} {|c_{l, \epsilon_0}||c_{m, \epsilon_0}|}{|l+m|^2}\|a^2_{(k,q)}\|_{L^\infty H^4}\nonumber\\
    \lesssim&K^{12}N^3_{\Lambda}\sum_{m,l \in\ZZ\backslash\{0\}, m+l\neq 0} |l+m|^{-2}\lambda^{5+4\alpha}_q\ell^{-8}_q\lambda^{4+9\epsilon_0}_{q+1} \nonumber\\
\lesssim&K^{12}N^3_{\Lambda}\lambda^{5+4\alpha}_q\ell^{-8}_q\lambda^{4+9\epsilon_0}_{q+1} .\nonumber
\end{align}
Taking advantage of \eqref{wpq+wcq}, Proposition \ref{guji1} and Proposition \ref{est-ak}, we have
\begin{align*}
\|\nabla(\wpq+\wcq)\|_{L^\infty_t W^{3,1}}
\lesssim&\lambda^{-1}_{q+1}\|a_{(k,q)}  g_{(2, \sigma_0, k)}(t)\psi'_{\lambda^{-\epsilon_0}_{q+1}}(\lambda^{1-\epsilon_0}_{q+1}N_{\Lambda}k\cdot x)\|_{L^\infty_t  W^{4,1}}\\
\lesssim&\lambda^{-1}_{q+1}\|g_{(2, \sigma_0, k)}\|_{L^\infty}\|a_{(k,q)}\|_{L^\infty_t H^4}\|\psi'_{\lambda^{-\epsilon_0}_{q+1}}(\lambda^{1-\epsilon_0}_{q+1}N_{\Lambda}k\cdot x)\|_{H^4(\TTT^3)}\\
\lesssim&K^{3/2}\ell^{-8}_q\lambda^{4}_{q+1}.
\end{align*}
By virtue of \eqref{e:ubq}, \eqref{estimate-wp} and \eqref{estimate-wc}, we obtain that
\begin{align}
&\|(\wpq+\wcq)\ootimes \vv_q+\vv_q\ootimes (\wpq+\wcq)\|_{L^\infty_t W^{3,1}}\nonumber\\
\lesssim&(\|\wpq\|_{L^\infty_t H^3}+\|\wcq\|_{L^\infty_t H^3})\|\vv_q\|_{L^\infty_t H^3}
\lesssim\lambda^5_q\lambda^{9/2}_{q+1}.\nonumber
\end{align}
By \eqref{e:ubq} and \eqref{estimate-wt}, we have
\begin{align}
\|\wtq\ootimes \ubql +\ubql\ootimes \wtq \|_{L^\infty_t W^{3,1}}
\lesssim&\|\wtq\|_{L^\infty_tH^3}\|\ubql\|_{L^\infty_tH^3}\lesssim\lambda^5_q\lambda^{9/2}_{q+1}.\nonumber
\end{align}
By Proposition \ref{wtq-H3} and Proposition \ref{estimate-wq+1}, one deduces that
{\begin{align*}
    & \|\wpq\ootimes (\wcq+\wtq)\|_{L^\infty_t W^{3,1}}+\|\wcq\ootimes w_{q+1}\|_{L^\infty_t W^{3,1}}+\|\wtq\ootimes(\wpq+\wcq)\|_{L^\infty_t W^{3,1}}\\
     \le&\|\wcq\|_{L^\infty_tL^2}(\|\wpq\|_{L^\infty_tH^3}+\|\wcq\|_{L^\infty_tH^3}+\|\wtq\|_{L^\infty_tH^3})\\
     &+\|\wcq\|_{L^\infty_tH^3}(\|\wpq\|_{L^\infty_tL^2}+\|\wcq\|_{L^\infty_tL^2}+\|\wtq\|_{L^\infty_tL^2})\\
&+\|\wtq\|_{L^\infty_tL^2}\|\wpq\|_{L^\infty_tH^3}+\|\wtq\|_{L^\infty_tH^3}\|\wpq\|_{L^\infty_tL^2}\\
\lesssim& C_0\ell^{-6}_q\delta^{1/2}_{q+1}\lambda^{9/2}_{q+1}+N_{\Lambda}K^{10}\ell^{-12}_q\lambda^{4+8\epsilon_0}_{q+1}.
\end{align*}}
Therefore, collecting these estimates together imply that
\begin{align*}
     \|\RR_{q+1}\|_{L^\infty_t W^{3,1}}\lesssim K^{12}\ell^{-12}_q\lambda^{5+\sigma_0-\frac{\epsilon_0}{2}}_{q+1}.
\end{align*}
Thanks to the condition \eqref{b-beta}, we prove \eqref{R—W}, so that we complete the proof of Proposition~\ref{R-q+1}.}
\end{proof}
\begin{proposition}\label{est-E}Let $u_{q+1}=\bar{u}_q+w_{q+1}$, we have
    \begin{align*}
\Big|E-\int_{\frac{3T}{4}}^T\int_{\R^3}|u_{q+1}|^2\dd x\dd t-3\delta_{q+2}\Big|\le \delta_{q+2}\lambda^{-\alpha}_q.
    \end{align*}
\end{proposition}
\begin{proof}We write
\begin{align*}
    \int_{\frac{3T}{4}}^T\int_{\R^3}|u_{q+1}|^2\dd x\dd t&=\int_{\frac{3T}{4}}^T\int_{\R^3}|\vv_q|^2\dd x\dd t+\int_{\frac{3T}{4}}^T\int_{\R^3}|w_{q+1}|^2\dd x\dd t+2\int_{\frac{3T}{4}}^T\int_{\R^3}\vv_q\cdot w_{q+1}\dd x\dd t\\
    &=:{\rm I+II+III}.
\end{align*}
For ${\rm III}$, we split $w_{q+1}$ into three parts $\wpq$, $\wcq$ and $\wtq$.  Note that
$$\supp_x \wpq=\supp_x\wcq\subseteq\Omega_{q+1},$$
we obtain
\begin{align*}
  |{\rm III}|\lesssim\Big|\int_{\frac{3T}{4}}^T\int_{\Omega_{q+1}}\vv_q\cdot \wpq\dd x\dd t\Big|+\Big|\int_{\frac{3T}{4}}^T\int_{\Omega_{q+1}}\vv_q\cdot \wcq\dd x\dd \Big|+\Big|\int_{\frac{3T}{4}}^T\int_{\R^3}\vv_q\cdot \wtq\dd x\dd t\Big|.
\end{align*}
Thanks to \eqref{e:vq-H3}, \eqref{e:v_ell-CN+1} and \eqref{estimate-wt}, one deduces that
\begin{align*}
\Big |\int_{\frac{3T}{4}}^T\int_{\R^3}\vv_q\cdot \wtq\dd x\dd t\Big|\le \|\vv_q\|_{L^\infty_tL^2}\|\wtq\|_{L^1_tL^2}\lesssim M\lambda^5_q\lambda^{-\frac{\sigma_0}{2}}_{q+1}.
\end{align*}
Moreover, by \eqref{est-wp-Lp} and \eqref{est-wcq-Lp}, we have
\begin{align*}
&\Big |\int_{\frac{3T}{4}}^T\int_{\Omega_{q+1}}\vv_q\cdot \wpq\dd x\dd t\Big|+\Big |\int_{\frac{3T}{4}}^T\int_{\Omega_{q+1}}\vv_q\cdot \wcq\dd x\dd t\Big|\\
\lesssim& K^{3}\|\vv_q\|_{L^\infty_{t,x}}(\|\wpq\|_{L^1_t L^\infty}+\|\wcq\|_{L^1_tL^\infty})\\
\lesssim& K^{3}\lambda^5_q\lambda^{-1+\frac{\epsilon_0}{2}}_{q+1}.
\end{align*}
By $w_{q+1}=\wpq+\wcq+\wtq$, we rewrite ${\rm II}$ as
\begin{align*}
{\rm II}=&\int_{\frac{3T}{4}}^T\int_{\R^3}|\wpq|^2\dd x\dd t+2\int_{\frac{3T}{4}}^T\int_{\Omega_{q+1}}\wpq\cdot(\wcq+\wtq)\dd x\dd t\\
&+\int_{\frac{3T}{4}}^T\int_{\R^3}(\wcq+\wtq)\cdot(\wcq+\wtq)\dd x\dd t.
\end{align*}
With the aid of  \eqref{estimate-wt}--\eqref{estimate-wp} and \eqref{est-wp-Lp}, one has
\begin{align*}
&\Big|\int_{\frac{3T}{4}}^T\int_{\R^3}\wpq\cdot(\wcq+\wtq)\dd x\dd t\Big|+\Big|\int_{\frac{3T}{4}}^T\int_{\R^3}(\wcq+\wtq)\cdot(\wcq+\wtq)\dd x\dd t\Big|\\
\le&\|\wpq\|_{L^2_{t,x}}\|\wcq\|_{L^2_{t,x}}+K^{3/2}\|\wpq\|_{L^1_tL^\infty}\|\wtq\|_{L^\infty_tL^2}\\
&+(\|\wcq\|_{L^2_{t,x}}+\|\wtq\|_{L^2_{t,x}})(\|\wcq\|_{L^2_{t,x}}+\|\wtq\|_{L^2_{t,x}})\\
\lesssim&\delta^{1/2}_{q+1}\lambda^{-1+\frac{\epsilon_0}{2}}_{q+1}+K^{3/2}\lambda^{-1+\frac{\epsilon_0}{2}-\frac{\sigma_0}{2}}_{q+1}+\lambda^{-\sigma_0}_{q+1}.
\end{align*}
Noting that $\supp_x \wpq\subseteq [-\tfrac{K}{2}, \tfrac{K}{2}]^3:=D_K$, by the definition of $\wpq$, we deduce that
\begin{align}
&\int_{\frac{3T}{4}}^T\int_{\R^3}|\wpq|^2\dd x\dd t=\int_{\frac{3T}{4}}^T\int_{ D_K}\tr (\wpq\otimes\wpq)\dd x\dd t\nonumber\\
=&\int_{\frac{3T}{4}}^T\int_{D_K}\tr \Big(\sum_{k\in\Lambda}
 a^2_{(k,q)} g^2_{(2, \sigma_0, k)}(t)\bar{k}\otimes \bar{k}\Big)\dd x\dd t\nonumber\\
 &+\int_{\frac{3T}{4}}^T\int_{D_K}\tr \Big(\sum_{k\in\Lambda}\sum_{m,l\in\ZZ\backslash\{0\}, m+l\neq 0}
 a^2_{(k,q)}  g^2_{(2, \sigma_0, k)}c_{l, \epsilon_0}c_{m, \epsilon_0}e^{2\pi \ii \lambda^{1-\epsilon_0}_{q+1} (l+m) k\cdot x/K }\bar{k}\otimes \bar{k} \Big)\dd x\dd t.\nonumber
\end{align}
The first term on the right-hand side of the above inequality can be rewritten as
\begin{align*}
&\int_{\frac{3T}{4}}^T\int_{D_K}\tr \Big(\sum_{k\in\Lambda}a^2_{(k,q)}  g^2_{(2, \sigma_0, k)}(t)\bar{k}\otimes \bar{k}\Big)\dd x\dd t\\
 =&\int_{\frac{3T}{4}}^T\int_{D_K}\tr \Big(\sum_{k\in\Lambda}a^2_{(k,q)}  \bar{k}\otimes \bar{k}\Big)\dd x\dd t+\int_{\frac{3T}{4}}^T\int_{D_K}\tr \Big(\sum_{k\in\Lambda}a^2_{(k,q)}  \mathbb{P}_{\neq 0}(g^2_{(2, \sigma_0, k)})\bar{k}\otimes \bar{k}\Big)\dd x\dd t\\
 =&3{\rho}_q\int_{\frac{3T}{4}}^T\int_{D_K}\eta^2_q\chi_q\dd x\dd t+\int_{\frac{3T}{4}}^T\int_{D_K}\tr \Big(\sum_{k\in\Lambda}a^2_{(k,q)}  \mathbb{P}_{\neq 0}(g^2_{(2, \sigma_0, k)})\bar{k}\otimes \bar{k}\Big)\dd x\dd t\\
 =&{E-\int_{\frac{3T}{4}}^T\int_{\R^3}|\vv_q|^2\dd x\dd t-3\delta_{q+2}}+\int_{\frac{3T}{4}}^T\int_{D_K}\tr \Big(\sum_{k\in\Lambda}a^2_{(k,q)}  \mathbb{P}_{\neq 0}(g^2_{(2, \sigma_0, k)})\bar{k}\otimes \bar{k}\Big)\dd x\dd t,
\end{align*}
where we have used \eqref{rho-bar} and \eqref{rho-q} in the last equality. Thanks to  \eqref{def-h-sig}, we have
\begin{align*}
&\Big|\int_{\frac{3T}{4}}^T\int_{D_K}\tr \Big(\sum_{k\in\Lambda}a^2_{(k,q)}  \mathbb{P}_{\neq 0}(g^2_{(2, \sigma_0, k)})\bar{k}\otimes \bar{k}\Big)\dd x\dd t\Big|\\
 =&\Big|\tr \Big(\sum_{k\in\Lambda}\int_{D_K}\int_{\frac{3T}{4}}^Ta^2_{(k,q)}  \dd h_{\sigma_0}(t) \dd x \,\bar{k}\otimes \bar{k}\Big)\Big|\\
 \lesssim&\|h_{\sigma_0}(t)\|_{L^\infty_t}(\|a_{(k,q)}\|^2_{L^\infty_tL^2}+\|a_{(k,q)}\|_{L^\infty_tL^2}\|\partial_ta_{(k,q)}\|_{L^\infty_tL^2})\\
\lesssim&\ell^{-16}_q\lambda^{-\sigma_0}_{q+1}.
\end{align*}
With the aid of Proposition \ref{est-cmr} and integration by parts, one deduces that
\begin{align*}
&\Big|\int_{\frac{3T}{4}}^T\int_{D_K}\tr \Big(\sum_{k\in\Lambda}\sum_{m,l\in \ZZ\backslash\{0\}, m+l\neq 0}
a^2_{(k,q)} g^2_{(2, \sigma_0, k)}(t)c_{l, \epsilon_0}c_{m, \epsilon_0}e^{ \ii 2\pi\lambda^{1-\epsilon_0}_{q+1} N_{\Lambda}(l+m) k\cdot x/K }\bar{k}\otimes \bar{k} \Big)\dd x\dd t\Big|\\
\lesssim&K^4\sum_{m,l\in \ZZ\backslash\{0\}, m+l\neq 0}\|a^2_{(k,q)} \|_{L^\infty_tW^{1,\infty}}{|c_{l, \epsilon_0}| |c_{m, \epsilon_0}|}{\lambda^{-1+\epsilon_0}_{q+1}}\\
 \lesssim&K^{13}\sum_{m,l\in \ZZ\backslash\{0\}, m+l\neq 0}\lambda^{5+\alpha}_q\ell^{-6}_q\lambda^{-1+10\epsilon_0}_{q+1}l^{-4}m^{-4}\\
\lesssim&K^{13}\lambda^{5+\alpha}_q\ell^{-6}_q\lambda^{-1+10\epsilon_0}_{q+1}.
\end{align*}
In conclusion, we obtain that
\begin{align*}
\Big|E-\int_{\frac{3T}{4}}^T\int_{\R^3}|u_{q+1}|^2\dd x\dd t-3\delta_{q+2}\Big|
\lesssim K^{13}\ell^{-16}_q\lambda^{-\sigma_0}_{q+1},
\end{align*}
which together with \eqref{epsilon} and \eqref{b-beta} shows Proposition \ref{est-E}.
\end{proof}
\subsection{Iterative estimates at $q+1$ level}
Now we collect these estimates together to show that $u_{q+1}$ and the Reynolds  stress $\RRR_{q+1}$  satisfies \eqref{uq-tigh}--\eqref{e:RR_q-C0}.

Firstly, we define $u_{q+1}=\vv_q+w_{q+1}$. Owning to
\begin{align*}
    \vv_q=\ubql+\ubqnl\quad\text{and}\quad w_{q+1}=\wpq+\wcq+\wtq,
\end{align*}
 we give
\[\uqql=\ubql+\wpq+\wcq, \quad\uqqnl=\ubqnl+\wtq.\]
Using \eqref{e:ubql}--\eqref{e:ubq}, \eqref{estimate-wt}, Proposition \ref{wtq-H3}--\ref{estimate-wq+1}, we infer from \eqref{epsilon} and \eqref{b-beta} that
\begin{align*}
   &\|\uqql\|_{L^2_{t,x}\cap L^p_tL^{\infty}}\le \frac{M}{2}(1-\delta^{1/2}_q)+\frac{1}{2}\delta^{1/2}_{q+1}\le \frac{M}{2}(1-\delta^{1/2}_{q+1}),\\
   & \|\uqqnl\|_{\widetilde{L}^{\infty}_tB^{3}_{2,2}}  \le  \lambda^5_q+\lambda^{\frac{9}{2}}_{q+1}\le \lambda^5_{q+1},\\
   &\|\ubqnl\|_{\widetilde{L}^{\infty}_tB^{1/2}_{2,1}\cap \widetilde{L}^{1}_tB^{5/2}_{2,1} }\le M^{-1}+\sum_{m=2}^{q} \delta_{k+1}\lambda^{-6\alpha}_k +\delta_{q+2}\lambda^{-6\alpha}_{q+1}= M^{-1}+\sum_{m=2}^{q+1} \delta_{k+1}\lambda^{-6\alpha}_k,\\
   &\|u_{q+1}\|_{L^2_{t,x}}\leq M(1- \delta^{1/2}_q)+\frac{1}{2}\delta^{1/2}_{q+1}\le M(1- \delta^{1/2}_{q+1}),\\
   &\|u_{q+1}\|_{L^\infty_tH^3}  \le  \lambda^5_q+\lambda^{\frac 92}_{q+1}\le \lambda^{5}_{q+1}.
\end{align*}
Since $\supp_x \wpq=\supp_x\wcq\subseteq\Omega_{q+1}$ and $\supp_x\ubql=\Omega_{q}+[-\lambda^{-1}_q, \lambda^{-1}_q]^3\subseteq \Omega_{q+1}$, we have
\[\supp_x \uqql\subseteq\Omega_{q+1}.\]
Hence, we prove that estimates \eqref{uq-tigh}--\eqref{e:vq-H3} hold with $q$ replaced by $q+1$. Proposition \ref{est-E} directly yields  \eqref{e:E-q} at $q+1$ level.

Note that
\begin{align*}
    a_{(k,q)}(t)=0,\quad \text{for}\,\, 0\le t\le \frac{T}{4}+2\lambda^{-1}_{q-1},
\end{align*}
this fact combined with \eqref{supp-wtq} shows that, for $0\le t\le \frac{T}{4}+2\lambda^{-1}_{q-1}$,
\begin{align}\label{R-t}
\RR^{(1)}_{q+1}(t)=\RR^{(2)}_{q+1}(t)=\RR^{(3)}_{q+1}(t)=0, \,\,\wpq(t)=\wcq(t)=\wtq(t)=0.
\end{align}
Therefore, we obtain from the definition of $\RR_{q+1}$ in \eqref{def-R-q+1} that
\begin{align*}
\RR_{q+1}(t)=0, \quad 0\le t\le \frac{T}{4}+2\lambda^{-1}_{q-1}.
\end{align*}
This fact together with \eqref{suppx-R} gives \eqref{supp-Rq} at $q+1$ level. Proposition \ref{R-q+1} directly shows that \eqref{e:RR_q-C0} holds for $\RR_{q+1}$. Thanks to \eqref{ubq-t} and \eqref{R-t}, we obtain \eqref{uq+1=uq} by $u_{q+1}=\vv_q+w_{q+1}$.  By \eqref{uq-ubq} and \eqref{estimate-w}, we have
\begin{align*}
    \|u_{q+1}-u_q\|_{L^2_{t,x}\cap L^p_tL^\infty}\le &   \|u_q-\vv_q\|_{L^2_{t,x}\cap L^p_tL^\infty}+\|w_{q+1}\|_{L^2_{t,x}\cap L^p_tL^\infty}\\
    \le& \lambda^{-40}_q+\frac{1}{2}C_0\delta^{1/2}_{q+1}\le C_0\delta^{1/2}_{q+1},
\end{align*}
where the last inequality holds by \eqref{b-beta}, and thereby we give \eqref{uq+1-uq}. Therefore, we complete the proof of Proposition \ref{iteration}.
%%%%%%%%%%%%%%%%%%%%%%%%%%%%%%%%%%%%%%%%%%%%%%%%%%%%%%%%%%%%%%%%%%%%%%%%%%%%%%%%%%%%%%%%%%%%%%%%%%%%%%%%%%%%%%%%%%%%%%%%%%%%%%

\section*{Acknowledgement}
This work was supported by the National Key Research and Development Program of China (No. 2022YFA1005700).

\end{document}